\documentclass[11pt]{article}
\usepackage{amsmath, latexsym, amsfonts, amssymb, amsthm, amscd}
\usepackage{amsmath}
\usepackage{mathtools} 
\usepackage[left=2.4cm, right=2.5cm, top=2.5cm, bottom=2.5cm]{geometry}
\usepackage{upquote} 
\usepackage{color}
\usepackage{setspace} 
\usepackage[labelfont=bf]{caption} 
\usepackage{stmaryrd} 
\usepackage{multicol} 
\usepackage[dvipsnames]{xcolor}
\usepackage{lipsum}
\usepackage{kpfonts} 
\usepackage{dsfont} 
\usepackage{url}
\usepackage[english]{babel} 
\usepackage[utf8]{inputenc}
\usepackage[T1]{fontenc}
\usepackage{enumerate}
\usepackage{graphicx}
\usepackage{color,soul}
 \usepackage{comment}
\usepackage[toc,page]{appendix}
\usepackage[mathlines]{lineno}
\usepackage{etoolbox} 
\newcommand*\linenomathpatch[1]{%
	\cspreto{#1}{\linenomath}%
	\cspreto{#1*}{\linenomath}%
	\csappto{end#1}{\endlinenomath}%
	\csappto{end#1*}{\endlinenomath}%
}
\newcommand*\linenomathpatchAMS[1]{%
	\cspreto{#1}{\linenomathAMS}%
	\cspreto{#1*}{\linenomathAMS}%
	\csappto{end#1}{\endlinenomath}%
	\csappto{end#1*}{\endlinenomath}%
}
\expandafter\ifx\linenomath\linenomathWithnumbers
\let\linenomathAMS\linenomathWithnumbers
\patchcmd\linenomathAMS{\advance\postdisplaypenalty\linenopenalty}{}{}{}
\else
\let\linenomathAMS\linenomathNonumbers
\fi

\linenomathpatch{equation}
\linenomathpatchAMS{gather}
\linenomathpatchAMS{multline}
\linenomathpatchAMS{align}
\linenomathpatchAMS{alignat}
\linenomathpatchAMS{flalign}

\usepackage{bm} 

\usepackage{enumitem} 

\usepackage{xcolor}
\definecolor{Maroon}{HTML}{ad2231}
\definecolor{webgreen}{HTML}{008000}
\usepackage{dsfont}
\usepackage{hyperref}
\usepackage{lipsum}
\hypersetup{colorlinks, breaklinks, urlcolor=black, linkcolor=black, citecolor=webgreen} 
\allowdisplaybreaks
 \usepackage{marginnote}
  
\usepackage{amsthm}
\newtheorem{theorem}{Theorem}
\newtheorem{corollary}[theorem]{Corollary}
\newtheorem{proposition}[theorem]{Proposition}
\newtheorem{lemma}[theorem]{Lemma}
\newtheorem{remark}[theorem]{Remark}

\theoremstyle{definition}

\usepackage{todonotes}

\usepackage{chngcntr}
\newcommand{\email}[1]{\gdef\@email{\url{#1}}}

\newcommand{\R}{\mathbb{R}}

\newcommand{\ee}{\textnormal{e}}
\newcommand{\dd}{\textnormal{d}}
\newcommand{\Ff}{\mathcal{F}}
\newcommand{\Gg}{\mathcal{G}}
\newcommand{\Aa}{\mathcal{A}}
\newcommand{\Ww}{\mathcal{W}}
\newcommand{\uU}{\mathcal{U}}
\newcommand{\vV}{\mathcal{V}}

\newcommand{\1}{\mathbf{1}}

\newcommand{\W}{\mathbb{W}}
\newcommand{\Z}{\mathbb{Z}}
\newcommand{\PP}{\mathbb{P}}
\newcommand{\E}{\mathbb{E}}

\DeclareMathAlphabet{\mathpzc}{OT1}{pzc}{m}{it}

\begin{document}
	\title{Lévy processes under  level-dependent  Poissonian  switching }
	\author{
		Noah Beelders\footnote{Department of Mathematical Sciences,
			University of Liverpool, \texttt{N.Beelders@liverpool.ac.uk}}\,, \,   Lewis Ramsden\footnote{School for Businesses and Society, Univeristy of York, \texttt{lewis.ramsden@york.ac.uk}} \; \&   Apostolos D. Papaioannou\footnote{Department of Mathematical Sciences,
			University of Liverpool, \texttt{papaion@liverpool.ac.uk}}    
	}
	\maketitle
	
	
	\vspace{0.1in}
	
	\begin{abstract} 
	In this paper, we derive  identities for the upward and downward exit problems and resolvents for a process whose motion changes between two L\'evy processes if it is above (or below) a barrier $b$ and coincides with a Poissonian arrival time. This  can be expressed in the form of a (hybrid)   stochastic differential equation, for which the existence of its solution   is also discussed. All identities are given in terms of new generalisations of scale functions (counterparts of the scale functions from the theory of L\'evy processes). 
		To illustrate the applicability of  our results, the probability of ruin  is obtained for a risk process  with delays in the dividend payments. 
	\end{abstract}
	
	
	\section{Introduction}
	The refracted L\'evy process, first introduced in \cite{KL2010}, is defined as a strong solution to the stochastic differential equation (SDE) 
	\[V_t=X_t-\delta\int_{0}^{t}\mathbf{1}_{(V_s>b)}\dd s,\]
	where the driving noise $X$ is a spectrally negative L\'evy process (SNLP) and $b$, $\delta$ are positive constants. Since then, fluctuations of refracted L\'evy processes and their applications to insurance risk models with dividend payments have received a lot of attention, see \cite{CPRY2019, K2013,  KPP2014,  PY2017, PY2018,  F2014, WLL2023}, to mention a few. 
	
	New generalisations of  refracted L\'evy processes have been introduced in \cite{NY2019}, whose motions above and below $b$ are L\'evy processes, different from each other. In this case, the generalised refracted L\'evy process is a solution to the SDE
	\[L_t=L_0+\int_{0}^{t}\mathbf1_{(L_{t^-}\geq b)}\dd X_s+\int_{0}^{t}\mathbf1_{(L_{t}< b)}\dd Y_s,\]
	where $X$ and $Y$ are two independent spectrally negative L\'evy processes (SNLPs) with (possibly) different L\'evy exponents.As pointed out in \cite{NY2019}, in  the bounded variation  case a solution to the above SDE can be constructed by piecing  the excursions; otherwise we do not know the existence of a solution for the above SDE.
	
	In this paper,  we consider a further  extension  of the  generalised refracted L\'evy process in \cite{NY2019}, in which the switch between $X$ and  $Y$ does not occur when $b$ is crossed continuously, but instead when it is above $b$ and coincides with an arrival epoch of an independent Poisson process.
	Under this extension, the corresponding  process   $U = \{U_t\}_{t \geq 0}$ is  a solution to the (hybrid)  SDE
	\begin{equation} \label{eq:SDE-Main}
		U_t = U_0 + \int^t_0 \1_{\{U_{T_{N(s)}} \leq b\}} \dd X_s +\int^t_0 \1_{\{U_{T_{N(s)}} > b\}} \dd Y_s ,
	\end{equation}
	where $X$ and $Y$ are as above and independent of the Poisson process $N$  with arrival times $T_0 = 0$ and $\{T_i\}_{i\geq 1}$ (see Section  \ref{ConstructionofSwitching} for full details).   Utilising the Poisson arrival epochs, we   show that a pathwise  solution  exists to  Eq.~\eqref{eq:SDE-Main} (sometimes called a hybrid SDE, see \cite{AP2022, JWFX2020}), even in the unbounded variation case. The  aim of this paper is twofold. Firstly,  to  establish a set of identities for the   two sided exit problems and the potential measures (killed and non-killed) of $U$, written in terms of new  generalisations of  scale functions (related to the one and two sided exit problem scale functions of \cite{K2014}). Secondly, to briefly show the relevance of these identities in the context of applications for the ruin problem in risk theory.

	L\'evy processes  observed in Poisson arrival epochs have been introduced in \cite{AIZ2016}. Since then, several modifications  of Poisson arrival epoch points  in connection with L\'evy processes have been developed and have found numerous applications  in insurance risk models, see \cite{ LLWX2018,  LYD2023, L2021}. As such,  letting $X = \{X_t\}_{t \geq 0}$ and ${Y} = \{ {Y}_t := {X}_t - \delta t \}_{t \geq 0}$ for the proposed model, we obtain an  insurance risk process which has delays in the initiation and termination of dividend payments.  The justification of such a risk model occurs naturally since dividend payments in reality are made with delays and not at the exact moment that the surplus crosses some level $b$. 
	
	The  remainder of the paper  is structured as follows.  Section \ref{prelim}  recalls the basic  theory of scale functions and provides useful identities that will be used in the rest of the paper. We show that a solution to Eq.~\eqref{eq:SDE-Main} exists in Section \ref{ConstructionofSwitching} and discuss also the strong Markov property.   In Section  \ref{sec:main}, we define the generalised scale functions which are used to derive  identities for the two sided exit problem (exiting upwards above  level $a>0$ and downwards below level  0), the one-sided exit identities as well as the killed and non-killed potential measures. We lastly provide a brief  application of $U$ as a risk process by choosing $Y$ so that U reduces to a (refracted) risk model with delays in the dividend payments and subsequently derive an explicit expression for the probability of ruin. 
	
	\section{Preliminaries}
	\label{prelim}
	Let $X = \{X_{t}\}_{t \geq 0}$ be a SNLP defined on the filtered space $(\Omega, \mathcal{F}, \{\mathcal{F}_{t}\}_{t \geq 0}, \mathbb{P})$, where the filtration $\{\mathcal{F}_{t}\}_{t \geq 0}$ is assumed to satisfy the usual assumptions of right continuity and completion. We shall denote $\PP_x$ to be the probability measure given  the process  starts at $x$ and  $\E_x$ to be the associated expectation. When $x=0$, we shall drop the subscript. A L\'evy process with no positive jumps (the case of monotone paths is excluded) has its Laplace exponent $\psi(\vartheta):[0,\infty)\rightarrow \mathbb R$ defined as
	$
	\psi(\vartheta) := \log \E[e^{\vartheta X_1}]$,  
	which, by the L\'evy-Khintchine formula, has the form 
	
	\[ \psi(\vartheta) = \mu\vartheta + \frac{\vartheta^2 \sigma^2}{2} + 
	\int_{(-\infty, 0)}^{} \bigl(e^{\vartheta x} -1- \vartheta x \mathbf{1}_{\{x >- 1\}}\bigr) \nu(\mathrm{d}x),\]
	where $\mu \in \R$, $\sigma \geq 0$ 
	and $\nu$, the L\'evy measure, is a $\sigma$-finite measure 
	concentrated on $(-\infty,0)$ satisfying $\int_{(-\infty,0)}(1 \wedge |x|^2) \nu(\mathrm{d}x)  <\infty$. The above shows that $\psi$ is a continuous and strictly convex function, and that it tends to infinity as $\vartheta$ tends to infinity. Thus, for $q \geq 0$, one can define the right-inverse of the Laplace exponent
	$ \Phi_q := \sup\{\vartheta\geq 0: \psi(\vartheta) = q\}$,
	for which $\vartheta = 0$ is the unique solution to $\psi(\vartheta) = 0$ on $[0,\infty)$ if $\psi^{\prime}\left(0^{+}\right) \geq 0$ else there are two solutions.
	Further details about SNLPs can be found in the monographs of \cite{B1996, KK2012, K2014}.

	
	
	It is well-known that the fluctuation identities for $X$ rely heavily on the so-called $W$ and $Z$ \emph{scale functions} (see \cite[Chapter 8]{K2014}).  
	For any $q\geq 0$, we define $W^{(q)}: \R \to [0, \infty)$ to be the unique (up to a scaling constant), continuous increasing function with Laplace transform 
	\begin{equation}
		\int_{0}^{\infty} e^{-\vartheta x}W^{(q)}(x)\mathrm{d}x = \frac{1}{\psi_q(\vartheta)}, \;  \vartheta > \Phi_q, 
		\label{eq:LTofscaleW}
	\end{equation}
	where $\psi_q(\vartheta) := \psi(\vartheta) -q$ and $W^{(q)}(x)=0$ for $x<0$. In the rest of the paper, we write $W$ or $\psi$ instead of $W^{(0)}$ or $\psi_0$ for convenience. We define also  $Z^{(q)}: \mathbb{R} \rightarrow[1, \infty)$ having the form
	\[
	Z^{(q)}(x)=1+q \int_0^x W^{(q)}(y) \mathrm{d} y,
	\]
	and its bivariate generalisation $Z^{(q)}: \mathbb{R} \times [0,\infty) \rightarrow[1, \infty)$ having the form
	\begin{equation}
		Z^{(q)}(x,\theta)=\ee^{\theta x} \Bigl(1-\psi_q(\theta)  \int_0^x \ee^{-\theta y}W^{(q)}(y) \mathrm{d} y \Bigr), \label{eq:GeneralisedZFunc}
	\end{equation}
	where $Z^{(q)}(x,0) = Z^{(q)}(x)$ and $Z^{(q)}(x,\theta) = \ee^{\theta x}$ for $x \leq 0$. 
	With regards to the limits of scale functions, it is well-known (see, for instance, Eqs.~(2.21) and (2.13) in \cite{LLWX2018}) that 
	\begin{equation}
	\frac{W^{(q)}(a-x)}{W^{(q)}(a)}  \rightarrow  e^{-\Phi_q x}, \color{black} \; \quad  \frac{Z^{(q)}(a,\theta)}{W^{(q)}(a)} \rightarrow \frac{\psi_q(\theta)}{\theta - \Phi_q}, \quad \quad \text{ as } a \rightarrow \infty. \label{eq:LimitofRatioofScaleFuncs}
	\end{equation}
	In addition, the following useful identities for convolutions of the scale functions will be used throughout the paper. For any $p, q, x \geq 0$ and $p \neq q$, it holds that     
	\begin{align*}
		(p-q)\int_0^x W^{(p)}(x-y) W^{(q)}(y) \mathrm{d} y&=W^{(p)}(x)-W^{(q)}(x),  \notag \\
		(p-q)\int_0^x W^{(p)}(x-y) Z^{(q)}(y) \mathrm{d} y&=Z^{(p)}(x)-Z^{(q)}(x) .	\notag 
	\end{align*}
	The above identities introduced in \cite{LRZ2014} were used to derive a new class of scale functions, the so-called \emph{second generation scale functions}, with the aim of solving occupation time fluctuation identities. These will be used throughout the paper and have the following forms.  For $p, p+q \geq 0$ and $u, x \in \mathbb{R}$, we  define 	
	\begin{align} 
		\overline{W}_u^{(p, q)}(x)  := & \; W^{(p+q)}(x)-q \int_0^u W^{(p+q)}(x-y) W^{(p)}(y) \mathrm{d} y\notag \\ 
		= & \; W^{(p)}(x)+q \int_u^x W^{(p+q)}(x-y) W^{(p)}(y) \mathrm{d} y, \label{eq:SecondGenScaleFunc1}		
		\\
		\label{eq:SecondGenScaleFunc2}
		\overline{Z}_u^{(p, q)}(x)  := & \;  Z^{(p+q)}(x)-q \int_0^u W^{(p+q)}(x-y) Z^{(p)}(y) \mathrm{d} y\notag \\
		= &  \; Z^{(p)}(x)+q \int_u^x W^{(p+q)}(x-y) Z^{(p)}(y) \mathrm{d} y, 
	\end{align}  
	For the  SNLP $Y=\{Y_t\}_{t \geq 0}$ in Eq.~\eqref{eq:SDE-Main}, similar results as the above hold with the corresponding notation  $\mathbb{W}^{(p)}$ and $\mathbb{Z}^{(p)}$ for each $p \geq 0$ ($\overline{\W}^{(p,q)}$ and $\overline{\Z}^{(p,q)}$ for $p+q \geq 0$) which are interpreted as the counterparts of ${W}^{(p)}$ and ${Z}^{(p)}$ (resp.~$\overline{W}^{(p,q)}$ and $\overline{Z}^{(p,q)}$) associated with the SNLP $X$. Observe also that $Y = X$ yields $\W^{(p)} = W^{(p)}$ ($\overline{\W}^{(p,q)} = \overline{W}^{(p,q)}$) and similarly for $\Z^{(p)}$ ($\overline{\Z}^{(p,q)}$). Furthermore, the Laplace exponent of $Y$ will be denoted as $\psi_q^*(\vartheta) := \psi^*(\vartheta) - q$ with a corresponding right-inverse
	$\varphi_q = \sup \{\vartheta \geq 0: \psi^*(\vartheta) = q\}$.
	%
	\section{Pathwise solution}
	\label{ConstructionofSwitching}
	
	In this section, we discuss the existence of the solution of the SDE in Eq. \eqref{eq:SDE-Main} and show that is also has the strong Markov property. 
	
	Let  ${X}$ and ${Y} $ be  SNLPs starting from $x$. For the construction below, we shall consider $x=0$ (without the loss of generality), and a Poisson process $N := N(t)$ with arrival times $T_0 = 0$ and $T_i = \sum^i_k \xi_{\lambda,k}$, where $\{\xi_{\lambda,k}\}_{k\geq 1}$ is a sequence of i.i.d.~Exp($\lambda$) waiting times with $\lambda < \infty$. Furthermore, ${X}$,${Y}$ and $N$ are adapted to $\Ff_t$, and are independent. We note that $U$ in Eq.~\eqref{eq:SDE-Main}
	has the dynamics of $X$ when it is observed below the barrier $b$ and subsequently switches to the dynamics of $Y$ if it is simultaneously greater than $b$ and an arrival occurs.
	
	To show that such a process has a strong solution, we construct it pathwise. Hence, let the process start at some value $ U_0 = x$ and 
	define the random switching times $K^{-}_{b,0} = 0$, 
	\begin{align*}
		K^{+}_{b,n} &:= \min \{ T_i \geq K^{-}_{b,n-1} : x + X_{T_i}  -  X_{ K^{-}_{b,n-1}} +  \sum^{n-1}_{i=1} (Y_{K^{-}_{b,i}} - Y_{ K^{+}_{b,i}}) +  \sum^{n-1}_{i=1} (X_{K^{+}_{b,i}} - X_{ K^{-}_{b,i-1}}) > b \},  \\
		K^{-}_{b,n} &:= \min \{ T_i \geq K^{+}_{b,n} : x + Y_{T_i} - Y_{K^{+}_{b,n}} + \sum^{n}_{i=1} (X_{K^{+}_{b,i}} - X_{ K^{-}_{b,i-1}}) +  \sum^{n-1}_{i=1} (Y_{K^{-}_{b,i}} - Y_{ K^{+}_{b,i}})  \leq b \}, 
	\end{align*}
	for $n = 1,2,\dots$ in the above. It is clear from the above formulation that $K^{+}_{b,n} < K^{-}_{b,n}$ for $n = 1,2,\dots,$ which creates intervals over which the process switches between the dynamics of $X$ and $Y$. Thus, from the recursive times above, we can define the process as 
	\begin{equation} \label{eq:PathConstruction}
		U_t = \begin{cases}
			x + X_{t}  -  X_{ K^{-}_{b,n}} +  \sum^{n}_{i=1} (Y_{K^{-}_{b,i}} - Y_{ K^{+}_{b,i}}) +  \sum^{n}_{i=1} (X_{K^{+}_{b,i}} - X_{ K^{-}_{b,i-1}}),  & t \in [K^{-}_{b,n},K^{+}_{b,n+1}), n=0,1,2,\dots, \\
			x + Y_{t} - Y_{K^{+}_{b,n}} + \sum^{n}_{i=1} (X_{K^{+}_{b,i}} - X_{ K^{-}_{b,i-1}}) +  \sum^{n-1}_{i=1} (Y_{K^{-}_{b,i}} - Y_{ K^{+}_{b,i}}), &  t \in [K^{+}_{b,n},K^{-}_{b,n}), n=1,2,\dots
		\end{cases}
	\end{equation}
	Furthermore, we can write $K^{+}_{b,n}$ and $K^{-}_{b,n}$ as stopping times w.r.t.~the process $U$ defined above; i.e. given $K^{-}_{b,0} = 0$, we have for $n = 1,2,\dots$ that
	\begin{linenomath}
		\begin{equation*}
			K^{+}_{b,n} := \min \{ T_i \geq K^{-}_{b,n-1} : U_{T_i} > b \}, \quad \text{ and } \quad
			K^{-}_{b,n} := \min \{ T_i \geq K^{+}_{b,n} : U_{T_i} \leq b \}, 
		\end{equation*}
	\end{linenomath}
	which shows that these stopping times are adapted w.r.t.~$\Ff_t$. Thus, the stopping times are well-defined under this filtration. 
	
	Next, we shall show that the above formulation in Eq.~\eqref{eq:PathConstruction} corresponds pathwise to Eq.~\eqref{eq:SDE-Main}. Observe that $\1_{\{U_{T_{N(t)}} \leq b\}} =  1$ for $t \in [K^{-}_{b,n},K^{+}_{b,n+1})$ for all $n=0,1,2,\dots,$ and $0$ otherwise. Thus, for $U_0 = x$, we have for $t \in [K^{-}_{b,n},K^{+}_{b,n+1})$ and $n=0,1,2,\dots$ that
	\begin{align*}
		U_t =& \; 	x + X_{t}  -  X_{ K^{-}_{b,n}} +  \sum^{n}_{i=1} (Y_{K^{-}_{b,i}} - Y_{ K^{+}_{b,i}}) +  \sum^{n}_{i=1} (X_{K^{+}_{b,i}} - X_{ K^{-}_{b,i-1}}) \\
		=& \; x + \int^t_{ K^{-}_{b,n}} \1_{\{U_{T_{N(t)}} \leq b\}} \dd X_s  + \sum^n_{i=1} \int^{K^{+}_{b,i}}_{K^{-}_{b,i-1}}\1_{\{U_{T_{N(s)}} \leq b\}} \dd X_s  + \sum^n_{i=1} \int^{K^{-}_{b,i}}_{K^{+}_{b,i}}\1_{\{U_{T_{N(s)}} > b\}} \dd Y_s \\
		=& \; U_0  + \int^t_0  \1_{(U_{T_{N(s)}} \leq b)} \dd X_s + \int^t_0  \1_{(U_{T_{N(s)}} > b)} \dd Y_s. 
	\end{align*}
	Using the same line of logic as above, the same can be proven for  $t \in [K^{+}_{b,n},K^{-}_{b,n})$ and $n=1,2,\dots$.  
	
	Noting that every compact interval has a finite number of arrivals, the number of times that the process switches in the interval $[0,t]$ is finite. Since this process switches between two well-defined SNLPs for every pair of subsequent stopping times, we have the following theorem.
	
	\begin{theorem} \label{thm:PathwiseConstruction}
		For $U_0 = x$, there exists a strong solution to Eq.~\eqref{eq:SDE-Main}. 
	\end{theorem}
	
	\begin{remark}\label{UniquenessOfConstruction}
		\begin{itemize}
			\item[\upshape{(i)}]  \upshape {Given the arrival times of the Poisson process, $T_i$, the pathwise solution guarantees that  there is a unique construction of $U$. \color{black} We note that  in Section \ref{Applications}, our choices for $X$ and $Y$ allow us to prove in full details pathwise uniqueness  which is a consequence of choosing two processes that have positive drifts such that the point $b$ is irregular for itself.
			
			\item[\upshape{(ii)}] Although Eq.\,\eqref{eq:SDE-Main} forms  a (hybrid)  SDE with discontinuous coefficients, the Poissonian mechanism (for finite $\lambda$) significantly simplifies  Eq.\,\eqref{eq:SDE-Main}, as the switching mechanism is only triggered a finite number of times in any given compact
				time interval,  in contrast to the potentially infinite number of switches in the classical refraction model. }
		\end{itemize} 
	\end{remark}

\begin{remark}\label{markov}\upshape{
	 It is clear that $U$ does not have the strong Markov property on its own. Defining the process
	\begin{equation*}
		Q_t:=\1_{\{U_{T_{N(t)}} > b\}}, \label{eq:Qt-definition}
	\end{equation*}
	allows us to rewrite Eq.\,\eqref{eq:SDE-Main} as
	\begin{linenomath}
		\begin{equation*}
			U_t = U_0  +  \int_0^t (1-Q_s )\mathrm{d} X_s + \int_0^t Q_s \mathrm{d} Y_s, \label{eq:SDE-WithQt}
		\end{equation*}
	\end{linenomath}
and for a  stopping time  $\{\tau < \infty\}$ and $t\geq 0$,  we have 
\begin{linenomath}
	\begin{equation*}
		U_{\tau+t} = U_\tau  +  \int_0^t (1-Q_{\tau+s} )\mathrm{d} X_s + \int_0^t Q_{\tau+s} \mathrm{d} Y_s. \label{eq:SDE-WithQt}
	\end{equation*}
\end{linenomath}
Now for $t<T_{N(\tau)+1}-\tau$, it follows that $Q_{\tau+t}=Q_\tau$ and $U_{\tau+t}=U_\tau+(1-Q_\tau)(X_{\tau+t}-X_\tau)+Q_\tau(Y_{\tau+t}-Y_\tau)$. Thus, $(U_{\tau+t}, Q_{\tau+t})$ depends only on $(U_{\tau}, Q_{\tau})$. 
\noindent At the next Poisson epoch, we have  $Q_{\tau+t}=Q_{T_{N(\tau)+1}}=\1_{\{U_{T_{N(\tau)+1}} > b\}}$, where 
\[U_{T_{N(\tau)+1}}=U_\tau+(1-Q_\tau)(X_{T_{N(\tau)+1}}-X_\tau)+Q_\tau(Y_{T_{N(\tau)+1}}-Y_\tau),\]
which also depends   on $(U_{\tau}, Q_{\tau})$. Finally, for   $t>T_{N(\tau)+1}-\tau$, we have 
\begin{linenomath}
	\begin{equation*}
		U_{\tau+t} = U_{T_{N(\tau)+1}}  +  \int_0^t (1-Q_{T_{N(\tau)+1}+s} )\mathrm{d} X_s + \int_0^t Q_{T_{N(\tau)+1}+s} \mathrm{d} Y_s. \label{eq:SDE-WithQt}
	\end{equation*}
\end{linenomath} Thus,  we can use a similar argument to  conclude that for any $t>0$ that $(U_{\tau+t}, Q_{\tau+t})$ depends only on $(U_{\tau}, Q_{\tau})$ and post-$\tau$ increments of $X$, $Y$ and $N$. Therefore, $(U_{t}, Q_{t})$ possesses the strong Markov property.} 
	\end{remark}
	\section{Main results}\label{sec:main}
	
	In this section, we derive fluctuation identities for $U$. More specifically, we shall introduce new generalisations of scale functions (in terms of the classical scale functions in \cite{K2014}) and derive identities for the  upwards and downwards exit problems, as well as the potential measure of $U$. 
	
	To do this, we let $U_0=x$, fix $b \geq 0$ and, for $a \in \mathbb{R}^+$, define the continuous and Poissonian first passage stopping times
	\begin{equation}
			\tau_{a,U}^{+(-)}:=\inf \left\{t>0: U_t>(<)\; a\right\}, \quad \text{ and } \quad {T}_{a,U}^{+ (-)}= \; \min \left\{T_i: {U}_{T_i}> (<) \; a \right\}, \notag
	\end{equation}	
	with the conventions inf $\varnothing=\infty$ and $\min \varnothing = \infty$, respectively, where $U$ in their subscripts indicate the underlying process that is considered. We point out that these subscripts $U$ may change (to $X$ and $Y$) in the rest of the paper, depending  on the underlying process used, without otherwise altering the notion of these stopping times. Clearly, it holds that $\tau_{a,U}^{+ (-)} \leq T_{a,U}^{+ (-)}$, and similar inequalities hold for the stopping times with corresponding subscripts  $X$ and $Y$.
	
	We aim to derive the two-sided exit results for 
	\[
	\E_x\left( e^{-q \tau_{a,U}^{+}} \1_{\{\tau_{a,U}^{+} < \tau_{0,U}^{-}\}} \right) \; \; \text{and} \;\;  \E_x\left( e^{-q \tau_{0,U}^{-}} \1_{\{\tau_{0,U}^{-} < \tau_{a,U}^{+}\}} \right).
	\]
	We emphasize that the exit times in the above are not exit times of the process $U$ observed at Poisson arrivals, but rather the standard exit times of the process $U$ which switches its dynamics at Poissonian times. We shall also show in Section \ref{Applications} that the above Laplace transforms of the upwards and downwards exit times can be used to derive the probability of ruin in a  risk model with delays on dividend payments. Finally, it is worth highlighting that the above exit time identities are comparable to the classical L\'evy fluctuation literature and generalise existing results, see for e.g.~\cite{WLL2023}.
	
	To derive our main fluctuation results for $U$, we require the identities given in the following lemma and corollary.
	
	\begin{lemma}\label{lem:PoissonPotentialsforResolvents}
		Let $0 \leq b \leq a$, $q \geq 0$ and $0 < \lambda < \infty$. Then the following identities hold.
		\begin{enumerate}
			\item[{\upshape{(i)}}] For $ x,y \in [0,a]$
			\begin{equation}
				\E_x \Big(\int^\infty_0 \ee^{-q t} \1_{\{X_t \in \dd y, \; t < T_{b,X}^{+} \wedge \tau_{a,X}^+ \wedge \tau_{0,X}^-\}} \dd t \color{black}\Big) = \Bigl( \frac{\overline{W}^{(q,\lambda)}_b(x)}{\overline{W}^{(q,\lambda)}_b(a)} \overline{W}_{b-y}^{(q,\lambda)}(a-y) - \overline{W}_{b-y}^{(q,\lambda)}(x-y) \Bigr) \dd y, \notag 
			\end{equation}
			
			\item[{\upshape{(ii)}}] For $ x,y \in [0,a]$
			\begin{equation}
				\E_x \Big(\int^\infty_0 \ee^{-q t} \1_{\{Y_t \in \dd y, \; t < T_{b,Y}^{-} \wedge \tau_{a,Y}^+ \wedge \tau_{0,Y}^-\}} \dd t \color{black} \Big) = \Bigl( \frac{\overline{\W}^{(q,\lambda)}_{x-b}(x)}{\overline{\W}^{(q,\lambda)}_{a-b}(a)} \overline{\W}_{a-b}^{(q,\lambda)}(a-y) - \overline{\W}_{x-b}^{(q,\lambda)}(x-y) \Bigr) \dd y, \notag 
			\end{equation}
		\end{enumerate}
		
	\end{lemma}
	
	\begin{proof} 
		\upshape{(i)} Let $R^{(q,\lambda)}(x,\dd y) = \E_x \big(\int^\infty_0 \ee^{-q t} \1_{\{X_t \in \dd y, \; t < T_{b,X}^{+} \wedge \tau_{a,X}^+ \wedge \tau_{0,X}^-\}}  \dd t \color{black} \big)$ and assume that $X$ has paths of bounded variation. Then, for $x \in [0,b)$, we have by conditioning on $\tau_{b,X}^+$ and the strong Markov property that
		\begin{align}
			R^{(q,\lambda)}(x,\dd y) &= \E_x \Big(\int^\infty_0 \ee^{-q t} \1_{\{X_t \in \dd y, \; t < \tau_{b,X}^{+} \wedge \tau_{0,X}^-\}} \dd t \color{black} \Big) + \E_x \big(\ee^{-q \tau_{b,X}^+} \1_{\{\tau_{b,X}^+ < \tau_{0,X}^-\}} \big) R^{(q,\lambda)}(b,\dd y) \notag \\
			&= \frac{W^{(q)}(x)}{W^{(q)}(b)}\Bigl( W^{(q)}(b-y) \1_{\{y \in [0,b)\}} \dd y + R^{(q,\lambda)}(b,\dd y)\Bigr) - W^{(q)}(x-y) \1_{\{y \in [0,b)\}} \dd y, \label{eq:Lemma(i)-SMPIdentity1}
		\end{align}
		where the last equality follows by using Eqs.~\eqref{eq:ClassicalExitfromAbove} and \eqref{eq:ClassicalKilledPotential} from the Appendix.
		
		Now, let $e_\lambda \sim \text{Exp}(\lambda)$ that is independent of all other random variables. For $x \in [b,a]$, observe that $T^+_{b,X} \wedge \tau_{b,X}^- \overset{d}{=} e_\lambda \wedge \tau_{b,X}^-$. Hence,  by conditioning on $\tau_{b,X}^-$ and using the strong Markov property, it follows that
		\begin{align}
			R^{(q,\lambda)}(x,\dd y) &= \E_x \Big(\int^\infty_0 \ee^{-q t} \1_{\{X_t \in \dd y, \; t < e_\lambda \wedge \tau_{a,X}^{+} \wedge \tau_{b,X}^-\}} \dd t \color{black} \Big) + \E_x \Big(\ee^{-q \tau_{b,X}^-} \1_{\{\tau_{b,X}^- < e_\lambda \wedge \tau_{a,X}^{+}\}} R^{(q,\lambda)}(X_{\tau_{b,X}^-},\dd y) \Big)  \notag \\
			&= \E_x \Big(\int^\infty_0 \ee^{-(q+\lambda) t} \1_{\{X_t \in \dd y, \; t < \tau_{a,X}^{+} \wedge \tau_{b,X}^-\}} \dd t \color{black} \Big) - \E_x \big(\ee^{-(q+\lambda) \tau_{b,X}^-} \1_{\{\tau_{b,X}^- < \tau_{a,X}^{+}\}} W^{(q)}(X_{\tau_{b,X}^-} - y) \big) \1_{\{ y \in [0,b)\}} \dd y \notag \\
			&\;\;\;\;\; + \frac{1}{W^{(q)}(b)} \E_x \big(\ee^{-(q+\lambda) \tau_{b,X}^-} \1_{\{\tau_{b,X}^- <\tau_{a,X}^{+}\}} W^{(q)}(X_{\tau_{b,X}^-}) \big) \times \Bigl( W^{(q)}(b-y) \1_{\{y \in [0,b)\}} \dd y + R^{(q,\lambda)}(b,\dd y)\Bigr) \notag \\
			&= \Bigl( \frac{W^{(q+\lambda)}(x-b)}{W^{(q+\lambda)}(a-b)}W^{(q+\lambda)}(a-y) - W^{(q+\lambda)}(x-y) \Bigr) \1_{\{ y \in [b,a]\}} \dd y \notag \\
			&\;\;\;\;\; - \Bigl( \overline{W}^{(q,\lambda)}_{b-y}(x-y) - \frac{W^{(q+\lambda)}(x-b)}{W^{(q+\lambda)}(a-b)} \overline{W}^{(q,\lambda)}_{b-y}(a-y) \Bigr) \1_{\{ y \in [0,b)\}} \dd y \notag \\
			&\;\;\;\;\; + \Bigl(\frac{\overline{W}^{(q,\lambda)}_{b}(x)}{W^{(q)}(b)} - \frac{W^{(q+\lambda)}(x-b)}{W^{(q+\lambda)}(a-b)} \frac{\overline{W}^{(q,\lambda)}_{b}(a)}{W^{(q)}(b)} \Bigr) \times \Bigl( W^{(q)}(b-y) \1_{\{y \in [0,b)\}} \dd y + R^{(q,\lambda)}(b,\dd y)\Bigr) \notag \\
			&= \Bigl( \frac{W^{(q+\lambda)}(x-b)}{W^{(q+\lambda)}(a-b)}\overline{W}_{b-y}^{(q,\lambda)}(a-y) - \overline{W}_{b-y}^{(q,\lambda)}(x-y) \Bigr) \dd y \notag \\
			&\;\;\;\;\; + \Bigl(\frac{\overline{W}^{(q,\lambda)}_{b}(x)}{W^{(q)}(b)} - \frac{W^{(q+\lambda)}(x-b)}{W^{(q+\lambda)}(a-b)} \frac{\overline{W}^{(q,\lambda)}_{b}(a)}{W^{(q)}(b)} \Bigr) \times \Bigl( W^{(q)}(b-y) \1_{\{y \in [0,b)\}} \dd y + R^{(q,\lambda)}(b,\dd y)\Bigr), \label{eq:Lemma(i)-SMPIdentity2}
		\end{align}
		where the second equality follows by substituting Eq.~\eqref{eq:Lemma(i)-SMPIdentity1}, the third equality follows by using Eqs.~\eqref{eq:ClassicalKilledPotential} and \eqref{eq:lem-RonnieOccupationDownJump} from the Appendix, and the last equality follows by observing that $\overline{W}^{(q,\lambda)}_{b-y}(x-y) = W^{(q+\lambda)}(x-y)$ for $y \in [b,a]$.
		
		Then, by putting $x = b$ in the above equation, we observe that $W^{(q+\lambda)}(0) \neq 0$ for $X$ having bounded variation and also that $\overline{W}^{(q,\lambda)}_{b-y}(b-y) = W^{(q)}(b-y)$ for $y \in [0,a]$ which yields
		\begin{align}
			R^{(q,\lambda)}(b,\dd y) &= \Bigl( \frac{W^{(q+\lambda)}(0)}{W^{(q+\lambda)}(a-b)}\overline{W}_{b-y}^{(q,\lambda)}(a-y) - {W}^{(q)}(b-y)\1_{\{y \in [0,b) \}} \Bigr) \dd y \notag \\
			&\;\;\;\;\; + \Bigl(1 - \frac{W^{(q+\lambda)}(0)}{W^{(q+\lambda)}(a-b)} \frac{\overline{W}^{(q,\lambda)}_{b}(a)}{W^{(q)}(b)} \Bigr) \times \Bigl( W^{(q)}(b-y) \1_{\{y \in [0,b)\}} \dd y + R^{(q,\lambda)}(b,\dd y)\Bigr), \notag
		\end{align}
		and hence that
		\begin{equation}
			R^{(q,\lambda)}(b,\dd y) = \Bigr(\frac{W^{(q)}(b)}{\overline{W}^{(q,\lambda)}_{b}(a)} \overline{W}^{(q,\lambda)}_{b-y}(a-y) - {W}^{(q)}(b-y)\1_{\{y \in [0,b) \}} \Bigr) \dd y. \notag
		\end{equation}
		Substituting the above quantity into Eqs.~\eqref{eq:Lemma(i)-SMPIdentity1} and \eqref{eq:Lemma(i)-SMPIdentity2} yields the desired result.

 To prove the unbounded variation case, we use strong approximation. 
First recall that there exists a sequence of bounded variation processes $\{(X_s^{(n)})_{s \geq 0}: n \geq 1\}$ that strongly approximates $X$; i.e.~that $\lim \limits_{n \rightarrow \infty} \sup_{0 \leq s \leq t}\bigl|  X_s - X_s^{(n)}\bigr| = 0$ for any $t >0$ a.s..
We denote $T_{b,X}^{+}(n) := \min \{T_i>0: X^{(n)}_{T_i} > b \}$, $\tau_{a,X}^{+}(n) := \inf \{t>0: X^{(n)}_{t} > a \}$ and $\tau_{0,X}^{-}(n) := \inf \{t>0: X^{(n)}_{t} < 0 \}$
the stopping times corresponding to each process $X^{(n)}$. Then, it holds (see \cite{LRZ2014} pp.~1421 -- 1422) for any time $t>0$, $\PP_x$-a.s.,\,that $\tau_{a,X}^{+}(n) \wedge t \rightarrow \tau_{a,X}^{+} \wedge t$ and $\tau_{0,X}^{-}(n) \wedge t \rightarrow \tau_{0,X}^{-} \wedge t$.
We now show that, $T_{b,X}^{+}(n) \wedge t \rightarrow T_{b,X}^{+} \wedge t$, $\PP_x$-a.s.~First recall that the processes $N$ (and thus every renewal time $T_i$), $X$ and $X^{(n)}$ are  independent for every $n \geq 1$. Thus, by using Proposition 15, p.~30 of \cite{B1996}, it can be shown that $\PP_x(X_{T_i} = b) = 0$. In addition, since the number of Poisson arrivals in $[0,t]$ are finite, we shall consider the finite set of indices $\mathcal{I}_t := \{i>1:T_i \leq t\}$. If $\mathcal{I}_t$ is empty, then $T_{b,X}^{+}(n) \wedge t = T_{b,X}^{+} \wedge t = t$, $\PP_x$-a.s.~and so their convergence holds trivially. Thus, we need to only consider the case where $\mathcal{I}_t$ is non-empty. Furthermore, since $\lim \limits_{n \rightarrow \infty} \sup_{0 \leq s \leq t\wedge T_i}|  X_s - X_s^{(n)}| \leq \lim \limits_{n \rightarrow \infty} \sup_{0 \leq s \leq t}|  X_s - X_s^{(n)}| = 0$, it converges pointwise. In addition, since $ 1_{\{x>b\}}$ is discontinuous only at $b$ and since $\PP_x(X_{T_i} = b) = 0$, we  have that $\1_{\{X_{T_i}^{(n)} > b\}} \rightarrow \1_{\{X_{T_i} > b\}}$, $\PP_x$-a.s.~for all $i \in \mathcal{I}_t$. From this it follows that there exists a random integer $N$ such that for all
$n \ge N$ these indicators are simultaneously equal for all
$i \in \mathcal I_t$. Hence, for all sufficiently large $n$, the set of
indices $i \in \mathcal I_t$ for which $X^{(n)}_{T_i} > b$ is identical to
the corresponding set for $X$.
Therefore, the first index at which the exceedance occurs is eventually the
same; that is,
\[
\min\{T_i \le t : X^{(n)}_{T_i} > b\}
=
\min\{T_i \le t : X_{T_i} > b\},
\]
for all sufficiently large $n$. If no such index exists, then both minima are
greater than $t$. Consequently,
\[
T_{b,X}^{+}(n)\wedge t \;\to\; T_{b,X}^{+}\wedge t,
\qquad \PP_x\text{-a.s.}
\]\color{black}

	Now,  a similar approximating procedure can be utilised as in \cite{LRZ2014} (pp.~1421 -- 1422) to show that the bounded variation  potential measure and  scale functions converge to that of the unbounded variation cases.
	\color{black}
		\\
		\newline
		\upshape{(ii)} The proof follows the same idea as that of \upshape{(i)}, and thus, for brevity, we state only the main identities that are needed. Hence, let $\widetilde{R}^{(q,\lambda)}(x, \dd y) = \E_x \Big(\int^\infty_0 \ee^{-q t} \1_{\{Y_t \in \dd y, \; t < T_{b,Y}^{-} \wedge \tau_{a,Y}^+ \wedge \tau_{0,Y}^-\}}  \dd t \color{black}  \Big)$ and consider $Y$ for bounded variation paths. Then, for $e_\lambda \sim \text{Exp}(\lambda)$ that is independent of all other random variables and $x \in [0,b)$, observe that $T_{b,Y}^- \wedge \tau_{b,Y}^+ \overset{\dd}{=} e_\lambda \wedge \tau_{b,Y}^+$, and so by conditioning on $\tau_{b,Y}^+$, using the strong Markov property and Eqs.~\eqref{eq:ClassicalExitfromAbove} and \eqref{eq:ClassicalKilledPotential} from the Appendix, we follow the same argument as that used to derive Eq.~\eqref{eq:Lemma(i)-SMPIdentity2} to get
		\begin{equation}
			\widetilde{R}^{(q,\lambda)}(x,\dd y) = \frac{\W^{(q+\lambda)}(x)}{\W^{(q+\lambda)}(b)}\Bigl( \W^{(q+\lambda)}(b-y) \1_{\{y \in [0,b)\}} \dd y + \widetilde{R}^{(q,\lambda)}(b,\dd y)\Bigr) - \W^{(q+\lambda)}(x-y) \1_{\{y \in [0,b)\}} \dd y. \label{eq:Lemma(ii)-SMPIdentity1}
		\end{equation}
		Now, for $x \in [b,a]$, we condition on $\tau_{b,Y}^-$, use the strong Markov property and a substitition of Eq.~\eqref{eq:Lemma(ii)-SMPIdentity1} along with Eqs.~\eqref{eq:ClassicalKilledPotential} and \eqref{eq:lem-RonnieOccupationDownJump} from the Appendix to get
		\begin{align}
			\widetilde{R}^{(q,\lambda)}(x,\dd y) &= \Bigl( \frac{\W^{(q)}(x-b)}{\W^{(q)}(a-b)}\overline{\W}_{b-y}^{(q+\lambda,-\lambda)}(a-y) - \overline{\W}_{b-y}^{(q+\lambda,-\lambda)}(x-y) \Bigr) \dd y \notag \\
			&\;\;\;\;\; + \Bigl(\frac{\overline{\W}^{(q+\lambda,-\lambda)}_{b}(x)}{\W^{(q+\lambda)}(b)} - \frac{\W^{(q)}(x-b)}{\W^{(q)}(a-b)} \frac{\overline{W}^{(q+\lambda,-\lambda)}_{b}(a)}{\W^{(q+\lambda)}(b)} \Bigr) \times \Bigl( \W^{(q+\lambda)}(b-y) \1_{\{y \in [0,b)\}} \dd y + \widetilde{R}^{(q,\lambda)}(b,\dd y)\Bigr). \label{eq:Lemma(ii)-SMPIdentity2}
		\end{align}
		Then, by putting $x = b$ in the above equation and  observing  that $\W^{(q)}(0) \neq 0$ for $Y$ having bounded variation and also that $\overline{\W}^{(q+\lambda,-\lambda)}_{b-y}(b-y) = \W^{(q+\lambda)}(b-y)$ for $y \in [0,a]$,   we can derive that
		\begin{equation}
			\widetilde{R}^{(q,\lambda)}(b,\dd y) = \Bigr(\frac{\W^{(q+\lambda)}(b)}{\overline{\W}^{(q+\lambda,-\lambda)}_{b}(a)} \overline{\W}^{(q+\lambda,-\lambda)}_{b-y}(a-y) - {\W}^{(q+\lambda)}(b-y)\1_{\{y \in [0,b) \}} \Bigr) \dd y. \notag
		\end{equation}
		Then, by substituting the above equation into Eqs.~\eqref{eq:Lemma(ii)-SMPIdentity1} and \eqref{eq:Lemma(ii)-SMPIdentity2}, and noticing also that $\overline{\W}^{(q+\lambda,-\lambda)}_{b}(x) = \overline{\W}^{(q,\lambda)}_{x-b}(x)$, we prove the identity for the bounded variation case. The unbounded variation case is proven by the approximation approach mentioned in the proof of \upshape{(i)}.
	\end{proof}
	
	\begin{corollary}
		\label{lem:PoissonPotentialsAndFluctuations}
		Let $0 \leq b \leq a$, $q \geq 0$ and $0 < \lambda < \infty$. Then the following identities hold:
		\begin{enumerate}
			\item[{\upshape{(i)}}] For $ x \in [0,a]$ and $ y \in [b,a]$,
			\begin{equation}
				\E_x \Big(\ee^{-q T_b^{+}} \1_{\{X_{T_{b,X}^{+}} \in \dd y, \; T_{b,X}^{+} < \tau_{a,X}^+ \wedge \tau_{0,X}^-\}} \Big) = \lambda \Bigl( \frac{\overline{W}^{(q,\lambda)}_b(x)}{\overline{W}^{(q,\lambda)}_b(a)} W^{(q+\lambda)}(a-y) - W^{(q+\lambda)}(x-y) \Bigr) \dd y, \notag 
			\end{equation}
			and
			\begin{equation}
				\E_x \Big(\ee^{-q\tau_{a,X}^+} \1_{\{\tau_{a,X}^+ < T_{b,X}^{+} \wedge \tau_{0,X}^-\}} \Big) = \frac{\overline{W}^{(q,\lambda)}_b(x)}{\overline{W}^{(q,\lambda)}_b(a)}. \notag
			\end{equation}
			\item[{\upshape{(ii)}}] For $x \in [0,a]$ and $y \in [0,b]$, 
			\begin{linenomath}
				\begin{equation} 
					\mathbb{E}_x \Big(\mathrm{e}^{-q T_{b,Y}^{-}} \mathbf{1}_{\{Y_{T_{b,Y}^{-}} \in \mathrm{d} y, T_{b,Y}^{-} < \tau_{a,Y}^{+} \wedge \tau_{0,Y}^{-}\}} \Big) =\lambda\Bigl(\frac{\overline{\mathbb{W}}_{x-b}^{(q, \lambda)}(x)}{\overline{\mathbb{W}}_{a-b}^{(q, \lambda)}(a)} \overline{\mathbb{W}}_{a-b}^{(q, \lambda)}(a-y)-\overline{\mathbb{W}}_{x-b}^{(q, \lambda)}(x-y)\Bigr) \mathrm{d} y, \notag 
				\end{equation}
			\end{linenomath}
			and
			\begin{linenomath}
				\begin{equation} 
					\mathbb{E}_x \Big(\mathrm{e}^{-q \tau_{a,Y}^{+}} \mathbf{1}_{\left\{\tau_{a,Y}^{+}< T_{b,Y}^{-} \wedge \tau_{0,Y}^{-}\right\}}\Big)=\frac{\overline{\mathbb{W}}_{x-b}^{(q, \lambda)}(x)}{\overline{\mathbb{W}}_{a-b}^{(q, \lambda)}(a)}. \notag
				\end{equation}
			\end{linenomath}
			
			\item[{\upshape{(iii)}}] For $ x \in [0,a]$,
			\begin{equation}
				\E^{}_x \left( \ee^{-q \tau_{0,Y}^-} \1_{\{\tau_{0,Y}^- < T_{b,Y}^{-} \wedge \tau_{a,Y}^+  \}} \right) = \; \E^{}_x \left( \ee^{-q \tau_{0,Y}^- - \lambda \int^{\tau_{0,Y}^-}_0 \1_{\{Y_s\in (0,b)\}} \dd s} \, \1_{\{\tau_{0,Y}^- < \tau_{a,Y}^+\}} \right) 
				= \; \overline{\Z}^{(q+\lambda,-\lambda)}_b(x) -  \frac{\overline{\W}^{(q,\lambda)}_{x-b}(x)}{\overline{\W}^{(q,\lambda)}_{a-b}(a)} \overline{\Z}^{(q+\lambda,-\lambda)}_b(a), \notag
			\end{equation}
			and 
			\begin{equation}
				\E^{}_x \left( \ee^{-q \tau_{0,X}^-} \1_{\{\tau_{0,X}^- < T_{b,X}^{+} \wedge \tau_{a,X}^+   \}} \right) 
				= \; \E^{}_x\left( \ee^{-q \tau_{0,X}^- - \lambda \int^{\tau_{0,X}^-}_0 \1_{\{X_s\in(b,a) \}}\dd s} \, \1_{\{\tau_{0,X}^- < \tau_{a,X}^+\}} \right) 
				= \; \overline{Z}^{(q,\lambda)}_{b}(x) - \frac{\overline{W}^{(q,\lambda)}_{b}(x)}{\overline{W}^{(q,\lambda)}_{b}(a)}\overline{Z}^{(q,\lambda)}_{b}(a). \notag
			\end{equation}
		\end{enumerate}	
	\end{corollary}
	
	\begin{proof}
		\textit{\upshape{(i)}} We use the same reasoning as that of Corollary 3.1 in \cite{LLWX2018}. Hence, by noticing that the probability that an observation is made in $(t,t+\dd t)$ is $\lambda \dd t$ and is independent of $X$, we have that $T_{b,X}^+$ satisfies
		\begin{equation*}
			\PP_x\big(T_{b,X}^+ \in \dd t, \; X_t \in [b,\infty)\big) = \lambda \PP_x\big(T_{b,X}^+ > t, \; X_t \in [b,\infty)\big) \dd t,
		\end{equation*}
		and so we use the above to find for $y \in [b,a]$ that
				\begin{align}
			\E_x \bigl( e^{-q T_{b,X}^+} \1_{\{X_{T_{b,X}^+} \in \dd y, T_{b,X}^+ < \tau_{a,X}^+ \wedge \tau_{0,X}^-\}} \bigr) &= \int^\infty_0 e^{-q t} \PP_x \bigl(X_t \in \dd y, \; t < \tau_{a,X}^+ \wedge \tau_{0,X}^-, \; T_{b,X}^+ \in \dd t \bigr) \notag \\
			&= \lambda \int^\infty_0 e^{-q t} \PP_x \bigl(X_t \in \dd y, \; t < T_{b,X}^+ \wedge \tau_{a,X}^+ \wedge \tau_{0,X}^- \bigr) \dd t \notag \\
			&= \lambda \Bigl( \frac{\overline{W}^{(q,\lambda)}_b(x)}{\overline{W}^{(q,\lambda)}_b(a)} W^{(q+\lambda)}(a-y) - W^{(q+\lambda)}(x-y) \Bigr) \dd y, \notag
		\end{align}
		where the last equality follows by using Lemma \ref{lem:PoissonPotentialsforResolvents} (i) and that $\overline{W}^{(q,\lambda)}_{b-y}(x-y)=W^{(q+\lambda)}(x-y)$ for $y \in [b,a]$.
		\newline
		\indent For the second identity, observe by conditioning on $T_{b,X}^+$, using the strong Markov property and then conditioning on $X_{T_{b,X}^+}$ that
		\begin{align}
			\E_x \big(\ee^{-q\tau_{a,X}^+} \1_{\{\tau_{a,X}^+ < T_{b,X}^{+} \wedge \tau_{0,X}^-\}} \big) &= \E_x \big(\ee^{-q\tau_{a,X}^+} \1_{\{\tau_{a,X}^+ < \tau_{0,X}^-\}} \big) - \E_x \big(\ee^{-q\tau_{a,X}^+} \1_{\{ T_{b,X}^{+} < \tau_{a,X}^+ < \tau_{0,X}^-\}} \big)\notag \\
			&= \E_x \big(\ee^{-q\tau_{a,X}^+} \1_{\{\tau_{a,X}^+ < \tau_{0,X}^-\}} \big) - \int^a_b \E_x \big(\ee^{-qT_{b,X}^+} \1_{\{ X_{T_{b,X}^+} \in \dd y, \; T_{b,X}^{+} < \tau_{a,X}^+ \wedge \tau_{0,X}^-\}} \big) \E_y \big(\ee^{-q\tau_{a,X}^+} \1_{\{\tau_{a,X}^+ < \tau_{0,X}^-\}} \big)  \notag \\
			&= \frac{W^{(q)}(x)}{W^{(q)}(a)} - \lambda\int^a_b \Bigl( \frac{\overline{W}^{(q,\lambda)}_b(x)}{\overline{W}^{(q,\lambda)}_b(a)} W^{(q+\lambda)}(a-y) - W^{(q+\lambda)}(x-y) \Bigr) \frac{W^{(q)}(y)}{W^{(q)}(a)} \dd y \notag \\
			&= \frac{1}{W^{(q)}(a)} \Bigl( W^{(q)}(x) - \frac{\overline{W}^{(q,\lambda)}_b(x)}{\overline{W}^{(q,\lambda)}_b(a)} \big[\overline{W}^{(q,\lambda)}_b(a) - W^{(q)}(a) \big] + \big[\overline{W}^{(q,\lambda)}_b(x) - W^{(q)}(x) \big] \Bigr) \notag \\
			&= \frac{\overline{W}^{(q,\lambda)}_b(x)}{\overline{W}^{(q,\lambda)}_b(a)}, \notag
		\end{align}
		where the second equality follows by using the first result of the proof along with Eq.~\eqref{eq:ClassicalExitfromAbove} from the Appendix, and the third equality uses Eq.~\eqref{eq:SecondGenScaleFunc1}.
		\\
		\newline
		\textit{\upshape{(ii)}} The result can be proven in a similar way to the above, but can also be seen directly from Theorem 1.2 in \cite{BPPR2016} or Corollary 3.2 in \cite{LLWX2018}.
		\\
		\newline
		\textit{\upshape{(iii)}} Observe from Remark 3.2 in \cite{AIZ2016} that
		\begin{equation}
			\E^{}_x \left( \ee^{-q \tau_{0,Y}^-} \1_{\{\tau_{0,Y}^- < T_{b,Y}^{-} \wedge \tau_{a,Y}^+  \}} \right) = \; \E^{}_x \left( \ee^{-q \tau_{0,Y}^- - \lambda \int^{\tau_{0,Y}^-}_0 \1_{(0,b)}(Y_s)\dd s} \, \1_{\{\tau_{0,Y}^- < \tau_{a,Y}^+\}} \right).  \notag
		\end{equation}
		Then, Theorem 1 in \cite{LRZ2014} yields that
		\begin{align}
			\; \E^{}_x &\left( \ee^{-q \tau_{0,Y}^- - \lambda \int^{\tau_{0,Y}^-}_0 \1_{(0,b)}(Y_s)\dd s} \, \1_{\{\tau_{0,Y}^- < \tau_{a,Y}^+\}} \right) = \overline{\Z}_0^{(q,\lambda)}(x) - \lambda \int^x_b \W^{(q)}(x-y) \overline{\Z}_0^{(q,\lambda)}(y) \dd y \notag \\
			&- \frac{\overline{\W}_0^{(q,\lambda)}(x) - \lambda \int^x_b \W^{(q)}(x-y) \overline{\W}_0^{(q,\lambda)}(y) \dd y}{\overline{\W}_0^{(q,\lambda)}(a) - \lambda \int^a_b \W^{(q)}(a-y) \overline{\W}_0^{(q,\lambda)}(y) \dd y} \Bigl(\overline{\Z}_0^{(q,\lambda)}(a) - \lambda \int^a_b  \W^{(q)}(a-y) \color{black} \overline{\Z}_0^{(q,\lambda)}(y) \dd y\Bigr) \notag \\
			&= \overline{\Z}^{(q+\lambda,-\lambda)}_b(x) -  \frac{\overline{\W}^{(q,\lambda)}_{x-b}(x)}{\overline{\W}^{(q,\lambda)}_{a-b}(a)} \overline{\Z}^{(q+\lambda,-\lambda)}_b(a), \notag
		\end{align}
		where the last equality holds by using Eqs.~\eqref{eq:SecondGenScaleFunc1} -- \eqref{eq:SecondGenScaleFunc2}.
		
		\noindent 
		The remaining identity can be shown in a similar way by first using Remark 3.2 in \cite{AIZ2016} to  observe that $\E^{}_x \left( \ee^{-q \tau_{0,X}^-} \1_{\{\tau_{0,X}^- < T_{b,X}^{+} \wedge \tau_{a,X}^+   \}} \right)
		= \; \E^{}_x\bigl( \ee^{-q \tau_{0,X}^- - \lambda \int^{\tau_{0,X}^-}_0 \1_{(b,a)}(X_s)\dd s} \, \1_{\{\tau_{0,X}^- < \tau_{a,X}^+\}} \bigr) $ and then using Theorem 1 in \cite{LRZ2014}.
	\end{proof}
	\subsection{Two-sided exit upwards and downwards}\label{Subsec:Two-Sided Exit}
	To derive the exit upwards and downwards, we are required to evaluate expectations that involve the Poissonian stopping times $T_{b,X}^{+}$ and $T_{b,Y}^{-}$. To be more precise, for a positive measurable (multivariate) function $f$, we will need to evaluate expectations of the form
	\[
	\E_x \Big(\ee^{-q T_{b,X}^{+}} \1_{\{ T_{b,X}^{+} < \tau_{a,X}^+ \wedge \tau_{0,X}^-\}} f(X_{T_{b,X}^{+}};z) \Big) \quad \text{ and } \quad \mathbb{E}_x \Big(\mathrm{e}^{-q T_{b,Y}^{-}} \mathbf{1}_{\{ T_{b,Y}^{-} < \tau_{a,Y}^{+} \wedge \tau_{0,Y}^{-}\}} f(Y_{T_{b,Y}^{-}};z) \Big).
	\]
	This can be done by first conditioning on $X_{T_{b,X}^+}$ $(Y_{T_{b,Y}^-})$ and then using Corollary \ref{lem:PoissonPotentialsAndFluctuations} (i) (\ref{lem:PoissonPotentialsAndFluctuations} (ii)) to get that
	\begin{align}
		\E_x \Big(\ee^{-q T_{b,X}^{+}} \1_{\{ T_{b,X}^{+} < \tau_{a,X}^+ \wedge \tau_{0,X}^-\}} f(X_{T_{b,X}^{+}};z) \Big) &= \int^a_b \E_x \Big(\ee^{-q T_{b,X}^{+}} \1_{\{X_{T_{b,X}^{+}} \in \dd y, \; T_{b,X}^{+} < \tau_{a,X}^+ \wedge \tau_{0,X}^-\}} \Big)f(y;z) \notag \\
		&= \frac{\overline{W}^{(q,\lambda)}_b(x)}{\overline{W}^{(q,\lambda)}_b(a)} \lambda \int^a_b W^{(q+\lambda)}(a-y) f(y;z) \dd y - \lambda \int^x_b W^{(q+\lambda)}(x-y) f(y;z) \dd y, \label{eq:Tb+ExpectationEvaluation}
	\end{align}
	and
	\begin{align}
		\mathbb{E}_x \Big(\mathrm{e}^{-q T_{b,Y}^{-}} \mathbf{1}_{\{ T_{b,Y}^{-} < \tau_{a,Y}^{+} \wedge \tau_{0,Y}^{-}\}} f(Y_{T_{b,Y}^{-}};z) \Big) &= \int^b_0 \mathbb{E}_x \Big(\mathrm{e}^{-q T_{b,Y}^{-}} \mathbf{1}_{\{Y_{T_{b,Y}^{-}} \in \mathrm{d} y, T_{b,Y}^{-} < \tau_{a,Y}^{+} \wedge \tau_{0,Y}^{-}\}} \Big)f(y;z) \notag \\
		&=\frac{\overline{\mathbb{W}}_{x-b}^{(q, \lambda)}(x)}{\overline{\mathbb{W}}_{a-b}^{(q, \lambda)}(a)} \lambda \int^b_0 \overline{\mathbb{W}}_{a-b}^{(q, \lambda)}(a-y) f(y;z) \dd y - \lambda \int^b_0 \overline{\mathbb{W}}_{x-b}^{(q, \lambda)}(x-y) f(y;z) \dd y, \label{eq:Vb-ExpectationEvaluation}
	\end{align}
	respectively. As a result, it is clear that the exit identities will contain integrals of the form
\[
	\lambda \int^x_b W^{(q+\lambda)}(x-y) f(y;z) \dd y \quad \text{ and } \quad \lambda \int^b_0 \overline{\mathbb{W}}_{x-b}^{(q, \lambda)}(x-y) f(y;z) \dd y,
\]
		and thus, for different choices of the function $f$, let us define auxiliary functions (containing the integrals above) in order to formulate our results more concisely. For $q,\lambda,x,u,b,z \geq 0$, let
\begin{align}
		\gamma_{b}^{(q,\lambda)}(x;z) &= W^{(q)}(x-z) - \overline{\W}^{(q,\lambda)}_{x-b}(x-z) + \lambda \int_0^{b-z} \overline{\mathbb{W}}_{x-b}^{(q,\lambda)}(x-z-y) W^{(q)}(y) \mathrm{d} y, \label{eq:Thm1-GammaAuxiliaryFuncs}\\
		\alpha_b^{(q,\lambda)}(x) &= Z^{(q)}(x) - \overline{\mathbb{Z}}_{b}^{(q+\lambda,-\lambda)}(x) + \lambda \int_0^{b} \overline{\mathbb{W}}_{x-b}^{(q,\lambda)}(x-y) Z^{(q)}(y) \mathrm{d} y , \label{eq:Thm1-AlphaAuxiliaryFuncs}
	\end{align}
	and further let
	\begin{align}
		\mathcal{W}_u^{(q,\lambda)}(x;z) =& \; \overline{\mathbb{W}}_{x-b}^{(q,\lambda)}(x-z) + \lambda \int_u^{x}  W^{(q+\lambda)}(x-y) \overline{\mathbb{W}}_{y-b}^{(q,\lambda)}(y-z) \mathrm{d} y, \label{eq:Thm1-WConvolutionAuxiliaryFuncs} \\
		\mathcal{G}^{(q,\lambda)}_{u}(x;z) =& \;  \gamma_{b}^{(q, \lambda)}(x;z) +  \lambda \int^x_u  W^{(q+\lambda)}(x-y) \gamma_{b}^{(q, \lambda)}(y;z) \mathrm{d} y,  \label{eq:Thm1-GConvolutionAuxiliaryFuncs}\\
		\mathcal{A}^{(q,\lambda)}_{u}(x) =& \;  \alpha_b^{(q, \lambda)}(x) +  \lambda \int^x_u  W^{(q+\lambda)}(x-y) \alpha_b^{(q, \lambda)}(y) \mathrm{d} y. \label{eq:Thm1-AConvolutionAuxiliaryFuncs} 
	\end{align}
	We will use the convention that $\gamma_{b}^{(q,\lambda)}(x) := \gamma_{b}^{(q,\lambda)}(x;0)$,
	$\mathcal{W}^{(q,\lambda)}_{u} (x) :=  \mathcal{W}^{(q,\lambda)}_{u} (x;0)$  and $\mathcal{G}^{(q,\lambda)}_{u} (x) := \mathcal{G}^{(q,\lambda)}_{u} (x;0)$, and will regularly use that $\mathcal{W}_x^{(q,\lambda)}(x;z) = \overline{\mathbb{W}}_{x-b}^{(q,\lambda)}(x-z)$, $\mathcal{G}^{(q,\lambda)}_{x}(x;z) =  \gamma_{b}^{(q, \lambda)}(x;z)$ and 	$\mathcal{A}^{(q,\lambda)}_{x}(x) =  \alpha_b^{(q, \lambda)}(x)$. 
	
	\begin{theorem} \label{Thm:ExitFromAbove}
		For $q, \lambda \geq 0$ and $0 \leq x, b \leq a$,
		\begin{equation}
			\mathbb{E}_x\left(\ee^{-q \tau_{a,U}^{+}} \mathbf{1}_{\{\tau_{a,U}^{+}<\tau_{0,U}^{-}\}}\right)=\frac{\mathcal{U}_{b, a}^{(q, \lambda)}(x)}{\mathcal{U}_{b, a}^{(q, \lambda)}(a)} ,   \label{ThmMain:ExitFromAbove}
		\end{equation}
		where
		\begin{equation} 
			\mathcal{U}_{b, a}^{(q, \lambda)}(x;y) = W^{(q)}(x-y) - \1_{\{x  > b\}}\Big(\mathcal{G}_{x}^{(q, \lambda)}(x;y) -\frac{\mathcal{W}_{x}^{(q, \lambda)}(x)}{\mathcal{W}_{b}^{(q, \lambda)}(a)}  \mathcal{G}_{b}^{(q, \lambda)}(a;y) \Big). \label{eq:UScaleFunc}   
		\end{equation}	
            with the convention that $\uU_{b, a}^{(q, \lambda)}(x) = \uU_{b, a}^{(q, \lambda)}(x;0)$
	\end{theorem}
	
	
	\begin{proof}
		We first note that $U$ starts with either $X$ or $Y$ dynamics depending on its starting position. Thus, $ \E^{}_x \big( \ee^{-q \tau_{a,U}^+ } \1_{\{\tau_{a,U}^+ < \tau_{0,U}^-\}}\big)$ will be denoted as $p^X(x)$ for $x \in [0,b]$, and $p^Y(x)$ for $x \in (b,a]$.
		
		Now, suppose that $x \in [0,b]$. Using the strong Markov property and Eq.~\eqref{eq:ClassicalExitfromAbove} from the Appendix, we have that
		\begin{align}
			p^{ X }(x) &= \E^{}_x \bigl( \ee^{-q \tau_{b,U}^+ } \1_{\{\tau_{b,U}^+ < \tau_{0,U}^-\}} \E^{}_{U_{\tau_{b,U}^+}} \bigl(  \ee^{-q \tau_{a,U}^+ } \1_{\{\tau_{a,U}^+ < \tau_{0,U}^-\}} \bigr) \bigr)
			= \E^{}_x \bigl( \ee^{-q \tau_{b,X}^+ } \1_{\{\tau_{b,X}^+ < \tau_{0,X}^-\}} \bigr) p^{ X }(b) = \frac{W^{(q)}(x)}{W^{(q)}(b)} p^{ X }(b), \label{eq:Thm1-SMPIdentity1}
		\end{align}
		where the second last equality follows since $\{X_t, t < T_{b,X}^+\}$ and $\{U_t, t < T_{b,U}^+\}$ have the same distribution w.r.t. $\PP_x$ when $x \in [0,b]$, and by recalling that $\tau_{b,X}^+ \leq T_{b,X}^+$ and $\tau_{b,U}^+ \leq T_{b,U}^+$.
		
		Similarly, suppose that the process starts at $x \in (b,a]$, and notice that $\{Y_t, t < T_{b,Y}^-\}$ and $\{U_t, t < T_{b,U}^-\}$ have the same distribution w.r.t. $\PP_x$ for these $x$-values. Therefore, by conditioning on whether $T_{b,U}^-$ or $\tau_{a,U}^+$ occurs first and using again the strong Markov property along with Corollary \ref{lem:PoissonPotentialsAndFluctuations} (ii), we get that
		\begin{align}
			p^{ Y }(x) &= \E^{}_x \left( \ee^{-q \tau_{a,U}^+} \1_{\{\tau_{a,U}^+ < T_{b,U}^{-} \wedge \tau_{0,U}^- \}} \right) + \E^{}_x \left( \ee^{-q \tau_{a,U}^+} \1_{\{T_{b,U}^{-} < \tau_{a,U}^+ <\tau_{0,U}^- \}} \right) \notag \\
			&= \E^{}_x \left( \ee^{-q \tau_{a,Y}^+} \1_{\{\tau_{a,Y}^+ < T_{b,Y}^{-} \wedge \tau_{0,Y}^-  \}} \right) + \E^{}_x \bigl( \ee^{-q T_{b,U}^{-}} \1_{\{T_{b,U}^{-} < \tau_{a,U}^+ \wedge \tau_{0,U}^-\}} \E^{}_{U_{T_{b,U}^{-}}} \bigl( \ee^{-q \tau_{a,U}^+ } \1_{\{\tau_{a,U}^+ < \tau_{0,U}^-\}} \bigr) \bigr) \notag \\
			&= \frac{\overline{\W}^{(q,\lambda)}_{x-b}(x)}{\overline{\W}^{(q,\lambda)}_{a-b}(a)} + \E^{}_x \bigl( \ee^{-q T_{b,Y}^{-}} \1_{\{T_{b,Y}^{-} < \tau_{a,Y}^+ \wedge \tau_{0,Y}^-\}} \E^{}_{Y_{T_{b,Y}^{-}}} \bigl( \ee^{-q \tau_{a,U}^+ } \1_{\{\tau_{a,U}^+ < \tau_{0,U}^-\}} \bigr) \bigr) \notag \\
			&= \frac{\Ww^{(q,\lambda)}_{x}(x)}{\Ww^{(q,\lambda)}_{a}(a)} + \frac{p^{ X }(b)}{W^{(q)}(b)} \E^{}_x \left( \ee^{-q T_{b,Y}^{-}} \1_{\{T_{b,Y}^{-} < \tau_{a,Y}^+ \wedge \tau_{0,Y}^-\}} W^{(q)}(Y_{T_{b,Y}^{-}}) \right), \label{eq:Thm1-SMPIdentity2}
		\end{align}
		where the last equality holds by using Eq.~\eqref{eq:Thm1-WConvolutionAuxiliaryFuncs} and substituting Eq.~\eqref{eq:Thm1-SMPIdentity1}. 
		
		From Eqs.~\eqref{eq:Thm1-SMPIdentity1} -- \eqref{eq:Thm1-SMPIdentity2}, it remains to derive $p^{ X }(b)$. To derive this quantity, we use a similar line of reasoning as for Eq.~\eqref{eq:Thm1-SMPIdentity2}, i.e.~by  conditioning on whether $T_{b,U}^+$ or $\tau_{a,U}^+$ occurs first and using the strong Markov property along with Corollary \ref{lem:PoissonPotentialsAndFluctuations} (i), 
		\begin{align}
			p^{ X }(b) =& \; \E^{}_b \left( \ee^{-q \tau_{a,U}^+} \1_{\{\tau_{a,U}^+ < T_{b,U}^{+} \wedge \tau_{0,U}^- \}} \right) + \E^{}_b \left( \ee^{-q \tau_{a,U}^+} \1_{\{T_{b,U}^{+} < \tau_{a,U}^+ < \tau_{0,U}^- \}} \right) \notag \\
			=& \; \frac{\overline{W}^{(q,\lambda)}_{b}(b)}{\overline{W}^{(q,\lambda)}_{b}(a)} + \E^{}_b \left( \ee^{-q T_{b,X}^{+}} \1_{\{T_{b,X}^{+} < \tau_{a,X}^+ \wedge \tau_{0,X}^-\}} p^{ Y }(X_{T_{b,X}^{+}})  \right) \notag \\
			=& \; \frac{\overline{W}^{(q,\lambda)}_{b}(b)}{\overline{W}^{(q,\lambda)}_{b}(a)} + \frac{1}{\Ww^{(q,\lambda)}_{a}(a)} \E^{}_b \bigl( \ee^{-q T_{b,X}^{+}} \1_{\{T_{b,X}^{+} < \tau_{a,X}^+ \wedge \tau_{0,X}^-\}} \Ww^{(q,\lambda)}_{X_{T_{b,X}^{+}}} (X_{T_{b,X}^{+}}) \bigr)\notag  \\
			& + \frac{p^{ X }(b)}{W^{(q)}(b)} \E^{}_b \bigl( \ee^{-q T_{b,X}^{+}} \1_{\{T_{b,X}^{+} < \tau_{a,X}^+ \wedge \tau_{0,X}^-\}} \E^{}_{X_{T_{b,X}^{+}}} \bigl( \ee^{-q T_{b,Y}^{-}} \1_{\{T_{b,Y}^{-} < \tau_{a,Y}^+ \wedge \tau_{0,Y}^-\}} W^{(q)}(Y_{T_{b,Y}^{-}}) \bigr) \bigr), \label{eq:Thm1-SMPIdentity3}
		\end{align}
		where the last line holds by substituting Eq.~\eqref{eq:Thm1-SMPIdentity2} into the second equality above.
		
		We now aim to evaluate the two expectations of the above equation. Using Eq.~\eqref{eq:Tb+ExpectationEvaluation} and \eqref{eq:Thm1-WConvolutionAuxiliaryFuncs},  
		\begin{equation}
			\E^{}_b \Bigl( \ee^{-q T_{b,X}^{+}} \1_{\{T_{b,X}^{+} < \tau_{a,X}^+ \wedge \tau_{0,X}^-\}} \Ww^{(q,\lambda)}_{X_{T_{b,X}^{+}}} (X_{T_{b,X}^{+}};y) \Bigr)
			=  \frac{\overline{W}_{b}^{(q,\lambda)}(b)}{\overline{W}_{b}^{(q,\lambda)}(a)} \Bigl( \Ww_{b}^{(q,\lambda)}(a;y) - {\Ww}_{a}^{(q,\lambda)}(a;y) \Bigr),  \label{eq:Thm1Proof-ExpTb+}
		\end{equation}
		for $y \geq 0$, and thus the first expectation in Eq. \eqref{eq:Thm1-SMPIdentity3} is given by the above for $y=0$.
		To evaluate the second expectation of Eq.~\eqref{eq:Thm1-SMPIdentity3}, first note that for $x > b$ from Eqs.~\eqref{eq:Vb-ExpectationEvaluation} and \eqref{eq:Thm1-GConvolutionAuxiliaryFuncs} that 
		\begin{align}
			\E^{}_{x} &\Bigl( \ee^{-q T_{b,Y}^{-}} \1_{\{T_{b,Y}^{-} < \tau_{a,Y}^+ \wedge \tau_{0,Y}^-\}} W^{(q)}(Y_{T_{b,Y}^{-}}) \Bigr) \notag \\
			=& \; \frac{\Ww_{x}^{(q,\lambda)}(x)}{\Ww_{a}^{(q,\lambda)}(a)} \Bigl( \Gg_{a}^{(q,\lambda)}(a) - W^{(q)}(a)  + \color{black} \Ww_{a}^{(q,\lambda)}(a)\Bigr) - \Bigl( \Gg_{x}^{(q,\lambda)}(x) - W^{(q)}(x)  + \color{black} \Ww_{x}^{(q,\lambda)}(x) \Bigr) \notag \\
			=& \; \frac{\Ww_{x}^{(q,\lambda)}(x)}{\Ww_{a}^{(q,\lambda)}(a)} \Bigl( \frac{\mathcal{W}_{a}^{(q, \lambda)}(a)}{\mathcal{W}_{b}^{(q, \lambda)}(a)}  \mathcal{G}_{b}^{(q, \lambda)}(a) -  \mathcal{U}_{b, a}^{(q, \lambda)}(a) \Bigr) -  \Bigl( \frac{\mathcal{W}_{x}^{(q, \lambda)}(x)}{\mathcal{W}_{b}^{(q, \lambda)}(a)}  \mathcal{G}_{b}^{(q, \lambda)}(a) -  \mathcal{U}_{b, a}^{(q, \lambda)}(x) \Bigr) \notag \\
			=& \; 	\mathcal{U}_{b, a}^{(q, \lambda)}(x) - \frac{\mathcal{W}_{x}^{(q, \lambda)}(x)}{\mathcal{W}_{a}^{(q, \lambda)}(a)} \mathcal{U}_{b, a}^{(q, \lambda)}(a) , \label{eq:Thm1Proof-ExpVb-}
		\end{align} 
		where the last equality follows by using  Eq.~\eqref{eq:UScaleFunc}.  Now to evaluate  the double expectation in \eqref{eq:Thm1-SMPIdentity3}, first note from Eq.~\eqref{eq:ConvolutionOfUbaFunc} in the Appendix that 
		\begin{equation}
			\lambda \int^a_b W^{(q+\lambda)}(a-y) \uU_{b,a}^{(q,\lambda)}(y) \dd y = \overline{W}^{(q,\lambda)}_{b}(a) - \uU_{b,a}^{(q,\lambda)}(a), \notag
		\end{equation}
		and second, using the above identity and Eq.~\eqref{eq:Tb+ExpectationEvaluation}, that
		\begin{equation}
			\E_b \Bigl(\ee^{-q T_{b,X}^{+}} \1_{\{T_{b,X}^{+} < \tau_{a,X}^+ \wedge \tau_{0,X}^-\}} \uU_{b,a}^{(q,\lambda)}(X_{T_{b,X}^{+}}) \Bigr) =  \frac{\overline{W}^{(q,\lambda)}_{b}(b)}{\overline{W}^{(q,\lambda)}_{b}(a)} \Bigl( \overline{W}^{(q,\lambda)}_{b}(a) - \uU_{b,a}^{(q,\lambda)}(a) \Bigl). \label{eq:Thm1Proof-ExpTb+Uba}
		\end{equation}
		Hence, the  second expectation of Eq.~\eqref{eq:Thm1-SMPIdentity3},  using Eqs  \eqref{eq:Thm1Proof-ExpTb+}--\eqref{eq:Thm1Proof-ExpTb+Uba}, turns out to be 
		\begin{align}
			\E^{}_b &\bigl( \ee^{-q T_{b,X}^{+}} \1_{\{T_{b,X}^{+} < \tau_{a,X}^+ \wedge \tau_{0,X}^-\}} \E^{}_{X_{T_{b,X}^{+}}} \bigl( \ee^{-q T_{b,Y}^{-}} \1_{\{T_{b,Y}^{-} < \tau_{a,Y}^+ \wedge \tau_{0,Y}^-\}} W^{(q)}(Y_{T_{b,Y}^{-}}) \bigr) \bigr)  \notag \\
			=& \; \E_b \bigl(\ee^{-q T_{b,X}^{+}} \1_{\{T_{b,X}^{+} < \tau_{a,X}^+ \wedge \tau_{0,X}^-\}} \uU_{b,a}^{(q,\lambda)}(X_{T_{b,X}^{+}}) \bigr)  - \frac{\mathcal{U}_{b, a}^{(q, \lambda)}(a)}{\mathcal{W}_{a}^{(q, \lambda)}(a)} 	\E^{}_b \bigl( \ee^{-q T_{b,X}^{+}} \1_{\{T_{b,X}^{+} < \tau_{a,X}^+ \wedge \tau_{0,X}^-\}} \Ww^{(q,\lambda)}_{X_{T_{b,X}^{+}}} (X_{T_{b,X}^{+}}) \bigr) \notag \\
			=& \;  \overline{W}_{b}^{(q,\lambda)}(b) -  \frac{\overline{W}_{b}^{(q,\lambda)}(b)}{\overline{W}_{b}^{(q,\lambda)}(a)}  \frac{\Ww_b^{(q,\lambda)}(a)}{\Ww_{a}^{(q,\lambda)}(a)} \uU_{b,a}^{(q,\lambda)}(a).
			\label{eq:Thm1-FundamentalBigEbIdentity} 
		\end{align}
		Then, by observing that $\overline{W}^{(q,\lambda)}_b(b) = W^{(q)}(b)$ and substituting the above equation and Eq.~\eqref{eq:Thm1Proof-ExpTb+} into Eq.~\eqref{eq:Thm1-SMPIdentity3}, we derive the desired quantity
		\begin{linenomath}
			\begin{equation} 
				p^{ X }(b) = \frac{ W^{(q)}(b)}{\uU^{(q,\lambda)}_{b,a}(a)}. \notag
			\end{equation}
		\end{linenomath}
		Finally, by substituting the above equation into Eq.~\eqref{eq:Thm1-SMPIdentity1}, we derive the result for $x \in [0,b]$. For $x \in (b,a]$,  we substitute $p^X(b)$ along with Eq.~\eqref{eq:Thm1Proof-ExpVb-} into Eq.~\eqref{eq:Thm1-SMPIdentity2} to get the required result.
		
	\end{proof}
	\begin{theorem} \label{Thm:ExitFromBelow}
		For $q, \lambda \geq 0$ and $0 \leq x, b \leq a$,
		\begin{equation}
			\mathbb{E}_x\left(\ee^{-q \tau_{0,U}^{-}} \mathbf{1}_{\left\{\tau_{0,U}^{-}<\tau_{a,U}^{+}\right\}}\right)= \mathcal{V}_{b, a}^{(q, \lambda)}(x) - \frac{\mathcal{U}_{b, a}^{(q, \lambda)}(x)}{\mathcal{U}_{b, a}^{(q, \lambda)}(a)} \mathcal{V}_{b, a}^{(q, \lambda)}(a)  , \label{ThmMain:ExitFromBelow}
		\end{equation}
		where $	\mathcal{U}_{b, a}^{(q, \lambda)}(x)$ is defined in Eq.~\eqref{eq:UScaleFunc} and
		\begin{equation} 
			\mathcal{V}_{b, a}^{(q, \lambda)}(x) = Z^{(q)}(x) - \1_{\{x > b\}}\Big(\mathcal{A}_{x}^{(q, \lambda)}(x) -\frac{\mathcal{W}_{x}^{(q, \lambda)}(x)}{\mathcal{W}_{b}^{(q, \lambda)}(a)}  \mathcal{A}_{b}^{(q, \lambda)}(a) \Big) \label{eq:VScaleFunc} .
		\end{equation}	
	\end{theorem}

	\begin{proof}
		Using a similar notation as in the proof of Eq.~\eqref{ThmMain:ExitFromAbove}, let $ \E^{}_x \big( \ee^{-q \tau_{0,U}^- } \1_{\{\tau_{0,U}^- < \tau_{a,U}^+\}}\big)$ be denoted by $g^X(x)$ for $x \in [0,b]$, and $g^Y(x)$ for $x \in (b,a]$.
		
		Now, suppose that $x \in [0,b]$. Then, conditioning on $\tau_{b,U}^+ $, using the strong Markov property  and Eqs.~\eqref{eq:ClassicalExitfromAbove} -- \eqref{eq:ClassicalExitfromBelow} of the Appendix, 
		\begin{align}
			g^{ X }(x) 
			=& \; \E^{}_x \bigl( \ee^{-q \tau_{0,U}^- } \1_{\{\tau_{0,U}^- < \tau_{b,U}^+\}} \bigr) + \E^{}_x \bigl( \ee^{-q \tau_{b,U}^+ } \1_{\{\tau_{b,U}^+ < \tau_{0,U}^-\}} \E^{}_{U_{\tau_{b,U}^+}} \bigl( \ee^{-q \tau_{0,U}^- } \1_{\{\tau_{0,U}^- < \tau_{a,U}^+\}}\bigr) \bigr) \notag \\
			=& \; \E^{}_x \bigl( \ee^{-q \tau_{0,X}^- } \1_{\{\tau_{0,X}^- < \tau_{b,X}^+\}} \bigr) + \E^{}_x \bigl( \ee^{-q \tau_{b,X}^+ } \1_{\{\tau_{b,X}^+ < \tau_{0,X}^-\}} \bigr) \E^{}_b \bigl( \ee^{-q \tau_{0,U}^- } \1_{\{\tau_{0,U}^- < \tau_{a,U}^+\}}\bigr) \notag \\
			=& \; Z^{(q)}(x) + \frac{W^{(q)}(x)}{W^{(q)}(b)} \bigl(g^X(b) - Z^{(q)}(b) \bigr), \label{eq:Thm2-SMPIdentity1}
		\end{align}
		where the second equality follows since $\{X_t, t < T_{b,X}^+\}$ and $\{U_t, t < T_{b,U}^+\}$ have the same distribution w.r.t. $\PP_x$ when $x \in [0,b]$, and by recalling that $\tau_{b,X}^+ \leq T_{b,X}^+$ and $\tau_{b,U}^+ \leq T_{b,U}^+$.
		
		Similarly, suppose the process starts at $x \in (b,a]$. Then, by observing that $\{Y_t, t < T_{b,Y}^-\}$ and $\{U_t, t < T_{b,U}^-\}$ have the same distribution w.r.t. $\PP_x$ for these $x$-values, we condition on $T_{b,U}^-$ and use the strong Markov property  to obtain
		\begin{align}
			g^{ Y }(x) &= \; \E^{}_x \left( \ee^{-q \tau_{0,U}^-} \1_{\{\tau_{0,U}^- < T_{b,U}^{-} \wedge \tau_{a,U}^+  \}} \right) + \E^{}_x \left( \ee^{-q \tau_{0,U}^-} \1_{\{T_{b,U}^{-} < \tau_{0,U}^- < \tau_{a,U}^+ \}} \right) \notag \\
			&= \; \E^{}_x \left( \ee^{-q \tau_{0,Y}^-} \1_{\{\tau_{0,Y}^- < T_{b,Y}^{-} \wedge \tau_{a,Y}^+ \}} \right) + \E^{}_x \left( \ee^{-q T_{b,Y}^{-}} \1_{\{T_{b,Y}^{-} < \tau_{a,Y}^+ \wedge \tau_{0,Y}^-\}} g^{ X }(Y_{T_{b,Y}^{-}}) \right) \notag \\
			&= \overline{\Z}^{(q+\lambda,-\lambda)}_b(x) -  \frac{\Ww^{(q,\lambda)}_{x}(x)}{\Ww^{(q,\lambda)}_{a}(a)} \overline{\Z}^{(q+\lambda,-\lambda)}_b(a) + \E^{}_x \left( \ee^{-q T_{b,Y}^{-}} \1_{\{T_{b,Y}^{-} < \tau_{a,Y}^+ \wedge \tau_{0,Y}^-\}} Z^{(q)}(Y_{T_{b,Y}^{-}}) \right) \notag \\
			& \quad  + \frac{(g^{ X }(b) - Z^{(q)}(b))}{W^{(q)}(b)} \E^{}_x \left( \ee^{-q T_{b,Y}^{-}} \1_{\{T_{b,Y}^{-} < \tau_{a,Y}^+ \wedge \tau_{0,Y}^-\}} W^{(q)}(Y_{T_{b,Y}^{-}}) \right), \label{eq:Thm2-SMPIdentity2A}
		\end{align}
		where the last line follows by using Corollary \ref{lem:PoissonPotentialsAndFluctuations} (iii) and Eq.~\eqref{eq:Thm2-SMPIdentity1}. Additionally, by using Eqs.~\eqref{eq:Vb-ExpectationEvaluation} and \eqref{eq:Thm1-AConvolutionAuxiliaryFuncs}, we observe that
		\begin{align}
		\E^{}_x &\left( \ee^{-q T_{b,Y}^{-}} \1_{\{T_{b,Y}^{-} < \tau_{a,Y}^+ \wedge \tau_{0,Y}^-\}} Z^{(q)}(Y_{T_{b,Y}^{-}}) \right) \notag \\
			=& \; \frac{\Ww_{x}^{(q,\lambda)}(x)}{\Ww_{a}^{(q,\lambda)}(a)} \Bigl( \Aa^{(q,\lambda)}_a(a) - Z^{(q)}(a) + \overline{\Z}_b^{(q+\lambda,-\lambda)}(a) \Bigr) - \Bigl( \Aa^{(q,\lambda)}_x(x) - Z^{(q)}(x) + \overline{\Z}_b^{(q+\lambda,-\lambda)}(x) \Bigr) \notag \\
			=& \; \mathcal{V}_{b, a}^{(q, \lambda)}(x) - \overline{\Z}_b^{(q+\lambda,-\lambda)}(x) - \frac{\Ww_{x}^{(q,\lambda)}(x)}{\Ww_{a}^{(q,\lambda)}(a)} \Bigl( \mathcal{V}_{b, a}^{(q, \lambda)}(a) - \overline{\Z}_b^{(q+\lambda,-\lambda)}(a) \Bigr),  \label{eq:EYxVB-ofZ(q)}
		\end{align}
		where the last equality holds by using Eq.~\eqref{eq:VScaleFunc}, and hence substituting the above equation into Eq.~\eqref{eq:Thm2-SMPIdentity2A} yields
		\begin{align}
			g^{ Y }(x) &= \mathcal{V}_{b, a}^{(q, \lambda)}(x) - \frac{\Ww_{x}^{(q,\lambda)}(x)}{\Ww_{a}^{(q,\lambda)}(a)}\mathcal{V}_{b, a}^{(q, \lambda)}(a) + \frac{(g^{ X }(b) - Z^{(q)}(b))}{W^{(q)}(b)} \E^{}_x \left( \ee^{-q T_{b,Y}^{-}} \1_{\{T_{b,Y}^{-} < \tau_{a,Y}^+ \wedge \tau_{0,Y}^-\}} W^{(q)}(Y_{T_{b,Y}^{-}}) \right), \label{eq:Thm2-SMPIdentity2}
		\end{align}
		From Eqs.~\eqref{eq:Thm2-SMPIdentity1} and \eqref{eq:Thm2-SMPIdentity2}, it suffices to derive $g^X(b)$. To do this, we condition on $T_{b,U}^+$, use Corollary \ref{lem:PoissonPotentialsAndFluctuations} (iii) and the strong Markov property to get 
		\begin{align}
			g^{ X }(b) =& \; \E^{}_b \left( \ee^{-q \tau_{0,U}^-} \1_{\{\tau_{0,U}^- < T_{b,U}^{+} \wedge \tau_{a,U}^+ \}} \right) + \E^{}_b \left( \ee^{-q \tau_{a,U}^+} \1_{\{T_{b,U}^{+} < \tau_{a,U}^+ < \tau_{0,U}^- \}} \right)\notag  \\
			=& \; \overline{Z}_b^{(q,\lambda)}(b) - \frac{\overline{W}_b^{(q,\lambda)}(b)}{\overline{W}^{(q,\lambda)}_{b}(a)}\overline{Z}^{(q,\lambda)}_{b}(a) + \E^{}_b \bigl( \ee^{-q T_{b,X}^{+}} \1_{\{T_{b,X}^{+} < \tau_{a,X}^+ \wedge \tau_{0,X}^-\}} g^{ Y }(X_{T_{b,X}^{+}})  \bigr). \notag
		\end{align}
		Substituting Eq.~\eqref{eq:Thm2-SMPIdentity2} into the expectation of the above equation, we get 
		\begin{align}
			g^{ X }(b)=& \; \overline{Z}^{(q,\lambda)}_b(b) - \frac{\overline{W}_b^{(q,\lambda)}(b)}{\overline{W}^{(q,\lambda)}_{b}(a)}\overline{Z}^{(q,\lambda)}_{b}(a) + \E^{}_b \bigl( \ee^{-q T_{b,X}^{+}} \1_{\{T_{b,X}^{+} < \tau_{a,X}^+ \wedge \tau_{0,X}^-\}} \vV^{(q,\lambda)}_{b,a}(X_{T_{b,X}^{+}}) \bigr)\notag  \\
			& - \frac{\vV^{(q,\lambda)}_{b,a}(a) }{\Ww^{(q,\lambda)}_{a}(a)} \E^{ }_b \bigl( \ee^{-q T_{b,X}^{+}} \1_{\{T_{b,X}^{+} < \tau_{a,X}^+ \wedge \tau_{0,X}^-\}} \Ww^{(q,\lambda)}_{X_{T_{b,X}^{+}}}(X_{T_{b,X}^{+}}) \bigr) \notag \\
			&+ \frac{(g^{ X }(b) - Z^{(q)}(b))}{W^{(q)}(b)}  \E^{ }_b \bigl( \ee^{-q T_{b,X}^{+}} \1_{\{T_{b,X}^{+} < \tau_{a,X}^+ \wedge \tau_{0,X}^-\}} \E^{}_{X_{T_{b,X}^{+}}} \color{black} \bigl( \ee^{-q T_{b,Y}^{-}} \1_{\{T_{b,Y}^{-} < \tau_{a,Y}^+ \wedge \tau_{0,Y}^-\}} W^{(q)}(Y_{T_{b,Y}^{-}})  \bigr) \bigr). \label{eq:Thm2-SMPIdentity3}
		\end{align} 
		We now need to compute only the first expectation of the above equation since the second and third expectations are known from Eqs.~\eqref{eq:Thm1Proof-ExpTb+} and \eqref{eq:Thm1-FundamentalBigEbIdentity}, respectively. Thus, by noticing from Eq.~\eqref{eq:ConvolutionOfVbaFunc} of the Appendix that
		\begin{equation}
			\lambda \int^a_b W^{(q+\lambda)}(a-y)     \vV_{b,a}^{(q,\lambda)}(y) \dd y = \overline{Z}^{(q,\lambda)}_{b}(a) - \vV_{b,a}^{(q,\lambda)}(a), \notag
		\end{equation}
		we have by using Eq.~\eqref{eq:Tb+ExpectationEvaluation} along with the above equation that
		\begin{equation}
			\E^{}_b \bigl( \ee^{-q T_{b,X}^{+}} \1_{\{T_{b,X}^{+} < \tau_{a,X}^+ \wedge \tau_{0,X}^-\}} \vV^{(q,\lambda)}_{b,a}(X_{T_{b,X}^{+}}) \bigr) =  \frac{\overline{W}^{(q,\lambda)}_{b}(b)}{\overline{W}^{(q,\lambda)}_{b}(a)} \Bigl( \overline{Z}^{(q,\lambda)}_{b}(a) - \vV_{b,a}^{(q,\lambda)}(a) \Bigl). \notag 
		\end{equation}
		Then, by observing that $\overline{Z}^{(q,\lambda)}_b(b) = Z^{(q)}(b)$ and $\overline{W}^{(q,\lambda)}_b(b) = W^{(q)}(b)$, we substitute Eqs.~\eqref{eq:Thm1Proof-ExpTb+}, \eqref{eq:Thm1-FundamentalBigEbIdentity} and the above equation into Eq.~\eqref{eq:Thm2-SMPIdentity3} to derive the desired quantity
		\begin{linenomath}
			\begin{equation} 
				g^{ X }(b) = Z^{(q)}(b) -  \frac{ W^{(q)}(b)}{\uU^{(q,\lambda)}_{b,a}(a)} \vV^{(q,\lambda)}_{b,a}(a). \notag
			\end{equation}
		\end{linenomath}
		Finally, by substituting the above equation into Eq.~\eqref{eq:Thm2-SMPIdentity1}, we derive the result for $x \in [0,b]$. For $x \in (b,a]$,  we substitute $g^X(b)$ along with Eq.~\eqref{eq:Thm1Proof-ExpVb-} into Eq.~\eqref{eq:Thm2-SMPIdentity2} to get the required result.
		%
		
	\end{proof}
	
	\begin{remark}\upshape{
		Let us assume that $X = Y$. Then, this assumption implies for the identities given in Proposition 2.1.~of \cite{WLL2023} that their refraction parameter $\delta = 0$, $\W^{(q)} = W^{(q)}$ and $\Z^{(q)} = Z^{(q)}$ which consequently yields 
		\begin{align*}
			\lambda \int^{b}_0 \overline{W}^{(q,\lambda)}_{x-b}(x-y)W^{(q)}(y) \dd y =& \; \overline{W}^{(q,\lambda)}_{x-b}(x) - W^{(q)}(x), \\
			\lambda \int^{b}_0 \overline{W}^{(q,\lambda)}_{x-b}(x-y)Z^{(q)}(y) \dd y =& \; \overline{Z}^{(q+\lambda,-\lambda)}_{b}(x) - Z^{(q)}(x).
		\end{align*}
    By the above two identities, $\gamma_b^{(q,\lambda)}(x) = \alpha_b^{(q,\lambda)}(x) = 0$ and therefore $\Gg_b^{(q,\lambda)}(x) = \Aa_b^{(q,\lambda)}(x) = 0$. Hence, we conclude that the case for $Y = X$ gives $\uU^{(q,\lambda)}_{b,a} = W^{(q)}(x)$ and $\vV^{(q,\lambda)}_{b,a} = Z^{(q)}(x)$ which reduces  Eqs.~\eqref{ThmMain:ExitFromAbove} and \eqref{ThmMain:ExitFromBelow} to those in Theorem \ref{Thm:ClassicalFluctuationIdentities} of the Appendix, the classical one-sided L\'evy fluctuation identities.}
	\end{remark}

\subsection{One sided exit upwards and downwards}

To derive the one-sided exit identities, we require the following lemma to determine the limits of the scale functions derived in Section \ref{Subsec:Two-Sided Exit}.

 \begin{lemma} \label{lem:LimitIdentities}
Let $q,\lambda > 0$, $a \in \R_+$, and $x,b \in [0,a]$. Then, for at least $\Phi_{q} > \varphi_{q+\lambda}$, the following limits are true:
		\begin{enumerate}
			\item[\upshape{(i)}] $\lim\limits_{a \rightarrow \infty} \W^{(q)}(a)/W^{(q)}(a) = 0$,
			\item[\upshape{(ii)}] $\lim\limits_{a \rightarrow \infty} \overline{\W}_{a-b}^{(q,\lambda)}(a)/W^{(q+\lambda)}(a) = 0$.
			\item[\upshape{(iii)}] $\lim\limits_{\theta \rightarrow \infty} \overline{\W}_{x-b}^{(q,\lambda)}(x+\theta)/W^{(q)}(\theta) = 0$.
		\end{enumerate}
	\end{lemma}
	
	\begin{proof}
		\upshape{(i)} Recall from \cite[Chapter 8]{K2014} that there exists a representation of the scale function for $q, x \geq 0$ such that 
		\begin{linenomath}
			$$
			W^{(q)}(x)=\mathrm{e}^{\Phi_q x} W_{\Phi_q}(x),
			$$
		\end{linenomath}
		where $W_{\Phi_q}(x)$ is the 0-scale function of the SNLP with Laplace exponent $\psi_{\Phi_q}(\theta):=\psi(\Phi_q+\theta)-q$. Furthermore, it is known (see for instance \cite{KK2012}) that
		\begin{linenomath}
			$$
			W_{\Phi_q}(\infty):=\lim _{x \rightarrow \infty} W_{\Phi_q}(x)=\frac{1}{\psi_{\Phi_q}^{\prime}(0+)}=\frac{1}{\psi^{\prime}(\Phi_q)},
			$$
		\end{linenomath}
		which implies that $W_{\Phi_q}(\infty)<\infty$ except if simultaneously $q=0$ and $\psi^{\prime}(0+)=0$. The same holds for the 0-scale function $\W_{\varphi_q}(x)$ having Laplace exponent $\psi^*_{\varphi_q}(\theta):=\psi^*(\varphi_q+\theta)-q$. Therefore, by noticing that $\Phi_q > \varphi_{q+\lambda}$ implies also that $\Phi_q > \varphi_q$, we have 
		\begin{linenomath}
			\begin{equation*}
				\lim\limits_{a \rightarrow \infty} \frac{\W^{(q)}(a)}{W^{(q)}(a)} = \lim\limits_{a \rightarrow \infty} e^{-(\Phi_q - \varphi_q)a} \frac{\W_{\varphi_q}(a)}{W_{\Phi_q}(a)} = 0.
			\end{equation*}
		\end{linenomath}
		\\ 
		\upshape{(ii)} From Eq.~\eqref{eq:SecondGenScaleFunc1}, we have that
		\begin{linenomath}
			\begin{equation*}
				\overline{\W}^{(q,\lambda)}_{a-b}(a) = \W^{(q)}(a) + \lambda \int^b_0 \W^{(q)}(a-y)\W^{(q+\lambda)}(y) \dd y,
			\end{equation*}
		\end{linenomath}
		Then, by observing that our assumption implies $\Phi_{q+\lambda} > \varphi_q$ and using a similar reasoning as for \upshape{(i)}, we get that
		\begin{linenomath}
			\begin{equation*}
				\lim\limits_{a \rightarrow \infty} \frac{\W^{(q)}(a)}{W^{(q+\lambda)}(a)} = \lim\limits_{a \rightarrow \infty} \ee^{-(\Phi_{q+\lambda} - \varphi_q)a} \frac{\W_{\varphi_q}(a)}{W_{\Phi_{q+\lambda}}(a)} = 0.
			\end{equation*}
		\end{linenomath}
		Therefore, since $W_{\Phi_{q+\lambda}}$ and $W_{\varphi_{q}}$ are continuous and bounded, and since $\ee^{-(\Phi_{q+\lambda} - \varphi_q)a}$ decreases as $a \rightarrow \infty$,  there exists some $C \in \R_+$ such that ~$\frac{\W^{(q)}(a)}{W^{(q+\lambda)}(a)} \leq C$ for all $a>0$. The required limit is thus derived by applying the dominated convergence theorem.\\
		\\
		\upshape{(iii)} The proof follows the same idea as that of \upshape{(ii)} by first noticing that the assumption $\Phi_q > \varphi_{q+\lambda}$ implies that
		\begin{equation*}
			\lim\limits_{\theta \rightarrow \infty} \frac{\W^{(q+\lambda)}(x+\theta)}{W^{(q)}(\theta)} = \lim\limits_{\theta \rightarrow \infty} \ee^{-(\Phi_{q} - \varphi_{q+\lambda})\theta} \frac{\ee^{\varphi_{q+\lambda}x}\W_{\varphi_{q+\lambda}}(x+\theta)}{W_{\Phi_{q}}(\theta)} = 0,
		\end{equation*}
		and then by using the dominated convergence theorem. 
	\end{proof}

 We now derive the one-sided exit identities

 \begin{proposition}  \label{label:Prop-OneSidedExitAbove}
		Let $0 < q,\lambda < \infty$. Then, for $0 \leq x,b \leq a$ and at least $\Phi_{q} > \varphi_{q+\lambda}$, we have
		\begin{linenomath}
			\begin{equation}
				\E_x\left( \ee^{-q \tau_{a,U}^+} \1_{\{ \tau_{a,U}^+ < \infty\}} \right) =  \frac{{\uU}_{b,a}^{(q,\lambda)\downarrow}(x)}{{\uU}_{b,a}^{(q,\lambda)\downarrow}(a)}, \notag
			\end{equation}
		\end{linenomath}
		where
  \begin{equation}
    \begin{aligned}
			\uU_{b,a}^{(q,\lambda)\downarrow}(x) =& \; \ee^{\Phi_q x}
			- \1_{\{x > b\}} \Bigl( \gamma_{b}^{(q,\lambda)\downarrow}(x) \\
			&- {\Z}^{(q)}(x-b;\varphi_{q+\lambda}) \frac{\gamma_{b}^{(q,\lambda)\downarrow}(a) + \lambda\int^{a}_b W^{(q+\lambda)}(a-u) \gamma_{b}^{(q,\lambda)\downarrow}(u) \dd u }{\Z^{(q)}(a-b;\varphi_{q+\lambda})+ \lambda \int^{a}_b W^{(q+\lambda)}(a-u) {\Z}^{(q)}(u-b;\varphi_{q+\lambda}) \dd u} \Bigr),\label{eq:LimitDown-Uba(x)}\\
		\end{aligned}
    \end{equation}
  
		and
		\begin{equation}
			\gamma_{b}^{(q,\lambda)\downarrow}(x) = \ee^{\Phi_q x} + \lambda \int_0^{\infty} \ee^{\Phi_q (b-y)}\overline{\mathbb{W}}_{x-b}^{(q,\lambda)}(x-b+y)\mathrm{d} y. \label{eq:LimitDown-GammaAuxiliaryFuncs}
		\end{equation}
	\end{proposition}
	
	\begin{proof}
		Using a level invariance argument and Theorem \ref{Thm:ExitFromAbove}, 
		\begin{equation}
			\E_x\left( \ee^{-q \tau_{a,U}^+} \1_{\{ \tau_{a,U}^+ < \infty\}} \right) = \lim\limits_{\theta \rightarrow \infty} \frac{\uU^{(q,\lambda)}_{b+\theta,a+\theta}(x+\theta)}{\uU^{(q,\lambda)}_{b+\theta,a+\theta}(a+\theta)}, \notag
		\end{equation}
		and so we derive the above limit. Hence, we notice that 
		\begin{equation}
			\lim\limits_{\theta \rightarrow \infty}\frac{\uU^{(q,\lambda)}_{b+\theta,a+\theta}(x+\theta)}{W^{(q)}(\theta)} = \lim\limits_{\theta \rightarrow \infty} \Bigl\{ \frac{W^{(q)}(x+\theta)}{W^{(q)}(\theta)} -\1_{\{x>b\}} \Bigl( \frac{\gamma^{(q,\lambda)}_{b+\theta}(x+\theta)}{W^{(q)}(\theta)} - \frac{\overline{\W}^{(q,\lambda)}_{x-b}(x+\theta)}{\Ww^{(q,\lambda)}_{b+\theta}(a+\theta)} \frac{\Gg^{(q,\lambda)}_{b+\theta}(a+\theta)}{W^{(q)}(\theta)}  \Bigr) \Bigr\}, \notag
		\end{equation}
		where
		\begin{align}
            \Ww^{(q,\lambda)}_{b+\theta}(a+\theta) &= \overline{\W}^{(q,\lambda)}_{a-b}(a+\theta) + \lambda \int^a_b W^{(q+\lambda)}(a-y) \overline{\W}^{(q,\lambda)}_{y-b}(y+\theta) \dd y, \notag \\
            \gamma^{(q,\lambda)}_{b+\theta}(x+\theta) &= W^{(q)}(x+\theta) - \overline{\W}_{x-b}^{(q,\lambda)}(x+\theta) + \lambda \int^{b+\theta}_{0} W^{(q)}(b+\theta-u) \overline{\W}^{(q,\lambda)}_{x-b}(x-b+u) \dd u, \notag \\
			\Gg^{(q,\lambda)}_{b+\theta}(a+\theta) &= \gamma^{(q,\lambda)}_{b+\theta}(a+\theta) + \lambda \int^a_b W^{(q+\lambda)}(a-u) \gamma^{(q,\lambda)}_{b+\theta}(u+\theta) \dd u, \notag
		\end{align}
		and the limit of each term needs to be determined.

		By using Eqs.~\eqref{eq:GeneralisedZFunc} -- \eqref{eq:SecondGenScaleFunc1} along with the dominated convergence theorem, 
		\begin{equation}
			\lim\limits_{\theta \rightarrow \infty} \frac{\overline{\W}^{(q,\lambda)}_{y-b}(y+\theta)}{\W^{(q+\lambda)}(\theta)} = \ee^{\varphi_{q+\lambda}b}\Z^{(q)}(y-b,\varphi_{q+\lambda}), \quad y \geq b, \label{eq:LimitofWbarTheta}
		\end{equation}
		and hence
		\begin{equation}
			\lim\limits_{\theta \rightarrow \infty} \frac{\Ww^{(q,\lambda)}_{b+\theta}(a+\theta)}{\W^{(q+\lambda)}(\theta)} = \ee^{\varphi_{q+\lambda}b} \Bigl(\Z^{(q)}(a-b,\varphi_{q+\lambda}) + \lambda \int^a_b W^{(q+\lambda)}(a-y) \Z^{(q)}(y-b,\varphi_{q+\lambda}) \dd y\Bigr). \label{eq:LimitofWwTheta}
		\end{equation}
  
		\noindent Then, for $\Phi_{q}>\varphi_{q+\lambda}$, we have by Eq.~\eqref{eq:LimitofRatioofScaleFuncs}, Lemma \ref{lem:LimitIdentities} (iii) and the dominated convergence theorem that 
		\begin{equation}
			\lim\limits_{\theta \rightarrow \infty} \frac{\gamma^{(q,\lambda)}_{b+\theta}(x+\theta)}{W^{(q)}(\theta)} = \ee^{\Phi_q x} + \lambda \int_0^{\infty} \ee^{\Phi_q (b-y)}\overline{\mathbb{W}}_{x-b}^{(q,\lambda)}(x-b+y)\mathrm{d} y = \gamma_{b}^{(q,\lambda)\downarrow}(x) \notag 
		\end{equation}
		exists, and hence that
		\begin{equation}
			\lim\limits_{\theta \rightarrow \infty} \frac{\Gg^{(q,\lambda)}_{b+\theta}(a+\theta) \color{black}}{W^{(q)}(\theta)} = \gamma^{(q,\lambda)\downarrow}_{b}(a) + \lambda \int^a_b W^{(q+\lambda)}(a-u) \gamma^{(q,\lambda)\downarrow}_{b}(u) \dd u. \notag 
		\end{equation}

		\noindent Since $ \lim\limits_{\theta \rightarrow \infty} \overline{\W}^{(q,\lambda)}_{x-b}(x+\theta) / \Ww^{(q,\lambda)}_{b+\theta}(a+\theta)$ is known already by using Eqs.~\eqref{eq:LimitofWbarTheta} -- \eqref{eq:LimitofWwTheta}, we use the two above limits along with Eq.~\eqref{eq:LimitofRatioofScaleFuncs} to conclude that 
  \begin{equation}
      \uU^{(q,\lambda)\downarrow}_{b,a}(x) = \lim\limits_{\theta \rightarrow \infty} \frac{\uU^{(q,\lambda)}_{b+\theta,a+\theta}(x+\theta)}{W^{(q)}(\theta)} \label{eq:ConclusionofUdownarrow(x)}
  \end{equation} 
  has the same form as Eq.~\eqref{eq:LimitDown-Uba(x)}.
		
	\end{proof}

	\begin{proposition}  \label{label:Prop-OneSidedExitDownwards}
		Let $0 < q,\lambda < \infty$. Then, for $0 \leq x,y,b \leq a$ and at least $\Phi_{q+\lambda} > \varphi_q$, we have
	       \begin{equation}
				\E_x\Big(e^{-q \tau_{0,U}^-} \1_{\{\tau_{0,U}^- < \infty \}} \Big) = \vV^{(q,\lambda)\uparrow}_{b}(x) -  \frac{\vV^{(q,\lambda)\uparrow}_{b}}{\uU^{(q,\lambda)\uparrow}_{b}(0)}\; \uU^{(q,\lambda)\uparrow}_{b}(x;0), \label{Prop:One-SidedExitBelow}
		\end{equation}
		where 
		\begin{align}
	       \uU_{b}^{(q,\lambda)\uparrow}(x;y) =& \; W^{(q)}(x-y)
			- \1_{\{x > b\}} \Bigl( \gamma_{b}^{(q,\lambda)}(x;y)  - \overline{\W}_{x-b}^{(q,\lambda)}(x) \frac{\int^{\infty}_b e^{-\Phi_{q+\lambda}u} \gamma_{b}^{(q,\lambda)}(u;y) \dd u }{\int^{\infty}_b e^{-\Phi_{q+\lambda}u} \overline{\W}_{u-b}^{(q,\lambda)}(u) \dd u} \Bigr)\label{eq:lim_a-of-Uba(x)}, \\ 
            \vV^{(q,\lambda)\uparrow}_{b}(x)
				=& \; Z^{(q)}(x) - \1_{\{x>b\}} \Bigl(\alpha_b^{(q,\lambda)}(x) - \overline{\W}_{x-b}^{(q,\lambda)}(x) \frac{\int^\infty_b \ee^{-\Phi_{q+\lambda}u} \alpha_b^{(q,\lambda)}(u) \dd u}{\int^\infty_b \ee^{-\Phi_{q+\lambda}u} \overline{\W}_{u-b}^{(q,\lambda)}(u) \dd u} \Bigr) , \label{eq:lim_a-of-Vba(x)} \\
            	\uU_{b}^{(q,\lambda)\uparrow}(y) =& \; \Z^{(q+\lambda)}(b-y,\varphi_q) - \lambda \int^{b-y}_0  \Z^{(q+\lambda)}(b-y-u,\varphi_q)W^{(q)}(u) \dd u \notag \\
			&+ \Z^{(q+\lambda)}(b,\varphi_q) \frac{\int^{\infty}_b e^{-\Phi_{q+\lambda}u} \gamma_{b}^{(q,\lambda)}(u;y) \dd u }{\int^{\infty}_b e^{-\Phi_{q+\lambda}u} \overline{\W}_{u-b}^{(q,\lambda)}(u) \dd u},\label{eq:lim_a-of-Uba(a)/W(q+lambda)}
		\end{align} 
		and
		\begin{equation}
			\begin{aligned}
				\vV^{(q,\lambda)\uparrow}_{b}
				=& \; \ee^{\varphi_q b} \Bigl(\frac{q}{\varphi_q} + \lambda \int^b_0 \ee^{-\varphi_q y}\Z^{(q+\lambda)}(y) \dd y \Bigr)  - \lambda \int^{b}_0  \Z^{(q+\lambda)}(b-u,\varphi_q)Z^{(q)}(u) \dd u \\
                &+ \Z^{(q+\lambda)}(b,\varphi_q) \frac{\int^\infty_b \ee^{-\Phi_{q+\lambda}u} \alpha_b^{(q,\lambda)}(u) \dd u}{\int^\infty_b \ee^{-\Phi_{q+\lambda}u} \overline{\W}_{u-b}^{(q,\lambda)}(u) \dd u}. \label{eq:lim_a-of-Vba(a)/W(q+lambda)}
			\end{aligned} 
		\end{equation}
  
	\end{proposition}

 \begin{remark}
    \upshape {The bivariate limits are proven here since they are the same as that needed for the one-sided potential measure in Theorem  \ref{label:Thm-PotentialMeasure}. }
 \end{remark}
	\begin{proof}[Proof of Proposition \ref{label:Prop-OneSidedExitDownwards}]
		We derive the desired identity by taking the limit of the two-sided exit downwards from Theorem \ref{Thm:ExitFromBelow} as $a \rightarrow \infty$. 
		Additionally, in some of the limits below, the dominated convergence theorem is applied since its usage is justified by noticing  from Eq.~\eqref{eq:LimitofRatioofScaleFuncs} that $W^{(q+\lambda)}(a-y)/W^{(q+\lambda)}(a) \rightarrow \ee^{-\Phi_{q+\lambda} y}$ as $a \rightarrow \infty$.
		
		Now, for $\Phi_{q+\lambda}>\varphi_q$, we have by Lemma \ref{lem:LimitIdentities} (ii) and the dominated convergence theorem that 	$\lim\limits_{a \rightarrow \infty} \gamma_{b}^{(q,\lambda)}(a;y)/W^{(q+\lambda)}(a) = \lim\limits_{a \rightarrow \infty} \alpha_{b}^{(q,\lambda)}(a)/W^{(q+\lambda)}(a) = 0$, and hence by using Eq.~\eqref{eq:LimitofRatioofScaleFuncs} that 
		\begin{align}
			\lim\limits_{a \rightarrow \infty} \frac{\Gg_{b}^{(q,\lambda)}(a;y)}{W^{(q+\lambda)}(a)}  =& \;  \lambda \color{black}\int^\infty_b \ee^{-\Phi_{q+\lambda} u} \gamma_{b}^{(q,\lambda)}(u;y) \dd u, \notag \\
            \lim\limits_{a \rightarrow \infty} \frac{\Aa_{b}^{(q,\lambda)}(a)}{W^{(q+\lambda)}(a)}  =& \;  \lambda \color{black} \int^\infty_b \ee^{-\Phi_{q+\lambda} u} \alpha_{b}^{(q,\lambda)}(u) \dd u, \notag \\
			\lim\limits_{a \rightarrow \infty} \frac{\Ww_{b}^{(q,\lambda)}(a;y)}{W^{(q+\lambda)}(a)}  =& \;  \lambda \color{black} \int^\infty_b \ee^{-\Phi_{q+\lambda} u} \overline{\W}_{u-b}^{(q,\lambda)}(u-y) \dd u. \label{eq:MiniLimitIdentities}
		\end{align}
	Thus, by using the above identities, we conclude that
 \begin{equation}
     \uU_{b}^{(q,\lambda)\uparrow}(x;y) := \lim\limits_{a \rightarrow \infty}\uU^{(q,\lambda)}_{b,a}(x;y), \quad \text{ and } \quad \vV_{b}^{(q,\lambda)\uparrow}(x) := \lim\limits_{a \rightarrow \infty}\vV^{(q,\lambda)}_{b,a}(x), \label{eq:ConclusionofUuparrow(x;y)}
 \end{equation}
 have the forms of Eqs.~\eqref{eq:lim_a-of-Uba(x)} and \eqref{eq:lim_a-of-Vba(x)}, respectively.
		
		Now, we notice that \begin{align}
        \lim\limits_{a \rightarrow \infty}\uU^{(q,\lambda)}_{b,a}(a;y)/\W^{(q)}(a) &= \lim\limits_{a \rightarrow \infty}\uU^{(q,\lambda)\uparrow}_{b}(a;y)/\W^{(q)}(a), \notag \\ 
        \lim\limits_{a \rightarrow \infty}\vV^{(q,\lambda)}_{b,a}(a)/\W^{(q)}(a) &= \lim\limits_{a \rightarrow \infty}\vV^{(q,\lambda)\uparrow}_{b}(a)/\W^{(q)}(a). \notag
        \end{align} 
        We derive these limits by using Eqs.~\eqref{eq:GeneralisedZFunc} and \eqref{eq:LimitofRatioofScaleFuncs} along with the dominated convergence theorem to observe that
		
			\begin{equation}
				\frac{\overline{\W}^{(q,\lambda)}_{a-b}(a-y)}{{\W}^{(q)}(a)} = \frac{\overline{\W}^{(q+\lambda,-\lambda)}_{b-y}(a-y)}{{\W}^{(q)}(a)}  \rightarrow \color{black} \ee^{-\varphi_q b} \Z^{(q+\lambda)}(b-y,\varphi_q), \quad \text{ as } a \rightarrow \infty, \label{eq:LimitofWBoldOverline}
                \end{equation}

                \begin{equation}
                 \frac{\overline{\Z}^{(q+\lambda,-\lambda)}_{b}(a)}{{\W}^{(q)}(a)}  \rightarrow \color{black} \frac{q}{\varphi_q}+\lambda \int^b_0 \ee^{-\varphi_q y}\Z^{(q+\lambda)}(y) \dd y, \quad \text{ as } a \rightarrow \infty.
			\end{equation}
		Thus, using the two above equations, that $\Gg^{(q,\lambda)}_{a}(a;y) = \gamma^{(q,\lambda)}_{b}(a;y)$, $\Aa^{(q,\lambda)}_{a}(a) = \alpha^{(q,\lambda)}_{b}(a)$ and the dominated convergence theorem, 
			\begin{equation}
				\lim\limits_{a \rightarrow \infty} \frac{1}{{\W}^{(q)}(a)}\bigl( W^{(q)}(a-y) -\Gg^{(q,\lambda)}_{a}(a;y)\bigr)= \ee^{-\varphi_q b} \Bigl( \Z^{(q+\lambda)}(b-y,\varphi_q) - \lambda \int^{b-y}_0  \Z^{(q+\lambda)}(b-y-u,\varphi_q)W^{(q)}(u) \dd u \Bigr), \label{eq:AuxLimitW-gamma}
			\end{equation}
        and
			\begin{equation}
				\lim\limits_{a \rightarrow \infty} \frac{1}{{\W}^{(q)}(a)}\bigl( Z^{(q)}(a) -\Aa^{(q,\lambda)}_{a}(a)\bigr) = \frac{q}{\varphi_q}+\lambda \int^b_0 \ee^{-\varphi_q y}\Z^{(q+\lambda)}(y) \dd y  - \lambda \ee^{-\varphi_q b} \int^{b}_0  \Z^{(q+\lambda)}(b-u,\varphi_q)Z^{(q)}(u) \dd u. \label{eq:AuxLimitW-alpha}
			\end{equation}
        
		\noindent Using Eq.~\eqref{eq:LimitofWBoldOverline} and the above two limits, we conclude that 
  \begin{equation}
  \uU_{b}^{(q,\lambda)\uparrow}(y) = \lim\limits_{a \rightarrow \infty}  \ee^{\varphi_q b} \cdot \; \uU^{(q,\lambda)}_{b,a}(a;y) / \W^{(q)}(a) \quad \text{ and } \quad \vV_{b}^{(q,\lambda)\uparrow} = \lim\limits_{a \rightarrow \infty}  \ee^{\varphi_q b} \cdot \; \vV^{(q,\lambda)}_{b,a}(a) / \W^{(q)}(a) \label{eq:ConclusionofUuparrow(y)}
  \end{equation}
 have the forms of Eqs.~\eqref{eq:lim_a-of-Uba(a)/W(q+lambda)} and \eqref{eq:lim_a-of-Vba(a)/W(q+lambda)}, respectively.
	\end{proof}

 \begin{remark}
   \upshape {  If the assumptions on the right-inverses $\Phi$ and $\varphi$ of the corresponding L\'evy exponents are relaxed, the limits yield indeterminate forms. These assumptions are hence imposed to ensure that appropriate limiting forms can be derived.}
 \end{remark}

	\subsection{Potential measures}
	
	In this subsection, we shall derive identities for the potential measure of $U$. We have first the following for the potential measure killed on  exiting $[0,a]$. 
	\begin{theorem}
		\label{label:Thm-PotentialMeasure}
		Let $0<b\leq a$ and $0<\lambda < \infty$. Then, for a Borel set $B \subseteq \R$, $q \geq 0$   we have 
		\begin{enumerate}
			\item[\upshape{(i)}] for $0 \leq x,y \leq a$, 
			\begin{equation}
				\E_x\Bigl( \int^\infty_0 e^{-q t} \1_{\{U_t \in B, \; t < \tau_{a,U}^+ \wedge \tau_{0,U}^-\}} \dd t \Bigr) = \int_{B\cap[0,a]\color{black}} \Bigl( \frac{\uU^{(q,\lambda)}_{b,a}(a;y)}{\uU^{(q,\lambda)}_{b,a}(a;0)}\uU^{(q,\lambda)}_{b,a}(x;0) - \uU^{(q,\lambda)}_{b,a}(x;y) \Bigr) \dd y, \label{eq:Cor1-PotentialMeasure}
			\end{equation}
			where $\uU^{(q,\lambda)}_{b,a}(x;y)$ is given in Eq.~\eqref{eq:UScaleFunc}.
			\item[\upshape{(ii)}] for $0 \leq x,b \leq a$, $y\geq 0$ and at least
			$\Phi_{q+\lambda} > \varphi_q$,
				\begin{equation}
					\E_x\left( \int^\infty_0 e^{-q t} \1_{\{U_t \in B, \; t < \tau_{0,U}^-\}} \dd t \right) = \int_{B\cap[0,\infty)\color{black}} \Bigl( \frac{{\uU}_{b}^{(q,\lambda)\uparrow}(y)}{\uU_{b}^{(q,\lambda)\uparrow}(0)} \uU_{b}^{(q,\lambda)\uparrow}(x;0) - \uU_{b}^{(q,\lambda)\uparrow}(x;y) \Bigr) \dd y, \notag
				\end{equation}
			where $\uU_{b}^{(q,\lambda)\uparrow}(x;y)$ and $\uU_{b}^{(q,\lambda)\uparrow}(y)$ are given by Eqs.~\eqref{eq:lim_a-of-Uba(x)} and \eqref{eq:lim_a-of-Uba(a)/W(q+lambda)}, respectively.
				\item[\upshape{(iii)}] for $0 \leq x,b \leq a$, $y\leq a$ and at least $\Phi_{q} > \varphi_{q+\lambda}$, 
				\begin{linenomath}
					\begin{equation}
						\E_x\left( \int^\infty_0 e^{-q t} \1_{\{U_t \in B, \; t < \tau_{a,U}^+\}} \dd t \right) = \int_{B\cap(-\infty,a]\color{black}} \Bigl( \frac{{\uU}_{b,a}^{(q,\lambda)\downarrow}(x)}{{\uU}_{b,a}^{(q,\lambda)\downarrow}(a)} {\uU}_{b,a}^{(q,\lambda)\downarrow}(a;y) - {\uU}_{b,a}^{(q,\lambda)\downarrow}(x;y) \Bigr) \dd y, \notag
					\end{equation}
				\end{linenomath}
				where 
				\begin{align}
					\uU_{b,a}^{(q,\lambda)\downarrow}(x;y) =& \; W^{(q)}(x-y)
					- \1_{\{x > b\}} \Bigl( \gamma_{b}^{(q,\lambda)}(x;y) \notag \\
					&- {\Z}^{(q)}(x-b;\varphi_{q+\lambda}) \frac{\gamma_{b}^{(q,\lambda)}(a;y) + \lambda\int^{a}_b W^{(q+\lambda)}(a-u) \gamma_{b}^{(q,\lambda)}(u;y) \dd u }{\Z^{(q)}(a-b;\varphi_{q+\lambda})+ \lambda \int^{a}_b W^{(q+\lambda)}(a-u) {\Z}^{(q)}(u-b;\varphi_{q+\lambda}) \dd u} \Bigr),\label{eq:LimitDown-Uba(x;y)}
				\end{align}
				and for which $\uU_{b,a}^{(q,\lambda)\downarrow}(x)$ is given in Eq.~\eqref{eq:LimitDown-Uba(x)}.
					\item[\upshape{(iv)}] for $0 \leq x,b \leq a$ and at least $\Phi_{q} > \varphi_{q+\lambda}$,
					\begin{linenomath}
						\begin{equation}
							\E_x\left( \int^\infty_0 e^{-q t} \1_{\{U_t \in B\}} \dd t \right) = \int_B \Bigl( \frac{\widetilde{{\uU}}_{b}^{(q,\lambda)}(y)}{\widetilde{{\uU}}_{b}^{(q,\lambda)}}\overline{{\uU}}_{b}^{(q,\lambda)}(x) - \overline{{\uU}}_{b}^{(q,\lambda)}(x;y) \Bigr) \dd y, \notag
						\end{equation}
					\end{linenomath}
					where 
					\begin{align}
						\overline{\uU}_{b}^{(q,\lambda)}(x;y) =& \; W^{(q)}(x-y)
						- \1_{\{x > b\}} \Bigl( \gamma_{b}^{(q,\lambda)}(x;y) - {\Z}^{(q)}(x-b;\varphi_{q+\lambda}) \frac{\int^{\infty}_b \ee^{-\Phi_{q+\lambda} u} \gamma_{b}^{(q,\lambda)}(u;y) \dd u }{\int^{\infty}_b \ee^{-\Phi_{q+\lambda} u} {\Z}^{(q)}(u-b;\varphi_{q+\lambda}) \dd u} \Bigr),\label{eq:LimitNatural-Uba(x;y)} \\
						\overline{\uU}_{b}^{(q,\lambda)}(x) =& \; \ee^{\Phi_q x}
						- \1_{\{x > b\}} \Bigl( \gamma_{b}^{(q,\lambda)\downarrow}(x) - {\Z}^{(q)}(x-b;\varphi_{q+\lambda}) \frac{\int^{\infty}_b \ee^{-\Phi_{q+\lambda} u} \gamma_{b}^{(q,\lambda)\downarrow}(u) \dd u }{ \int^{\infty}_b \ee^{-\Phi_{q+\lambda} u} {\Z}^{(q)}(u-b;\varphi_{q+\lambda}) \dd u} \Bigr),\label{eq:LimitNatural-Uba(x)} \\
						\widetilde{{\uU}}_{b}^{(q,\lambda)}(y) &= \Z^{(q+\lambda)}(b-y,\varphi_q) - \lambda \int^{b-y}_0  \Z^{(q+\lambda)}(b-y-u,\varphi_q)W^{(q)}(u) \dd u \notag \\
						&\;\;\;\;\; + \frac{\lambda}{\varphi_{q+\lambda} - \varphi_{q}} \frac{\int^{\infty}_b \ee^{-\Phi_{q+\lambda} u} \gamma_{b}^{(q,\lambda)}(u;y) \dd u }{\int^{\infty}_b \ee^{-\Phi_{q+\lambda} u} {\Z}^{(q)}(u-b;\varphi_{q+\lambda}) \dd u},\label{eq:LimitHarpoons-Ub(y)} \\
						\widetilde{{\uU}}_{b}^{(q,\lambda)} &=  - \lambda \int^{\infty}_0 \ee^{-\Phi_q (b-u)} \Z^{(q+\lambda)}(u,\varphi_q) \dd u + \frac{\lambda}{\varphi_{q+\lambda} - \varphi_{q}} \frac{\int^{\infty}_b \ee^{-\Phi_{q+\lambda} u} \gamma_{b}^{(q,\lambda)\downarrow}(u;y) \dd u }{\int^{\infty}_b \ee^{-\Phi_{q+\lambda} u} {\Z}^{(q)}(u-b;\varphi_{q+\lambda}) \dd u},\label{eq:LimitHarpoons-Ub}
					\end{align}
		\end{enumerate}
	\end{theorem}	
	\begin{proof}
	\upshape{(i)} 
		Using the same reasoning as in the previous section, $\E_x\big( \int^\infty_0 e^{-q t} \1_{\{U_t \in B, \; t < \tau_{a,U}^+ \wedge \tau_{0,U}^-\}} \dd t \big)$ will be denoted by $R^X(x,B)$ $(R^Y(x,B))$ for $x \in [0,b]$ $(x \in (b,a])$.
		
		For $x \in [0,b]$, we have by conditioning on $\tau_{b,U}^+$ and the strong Markov property  that 
		\begin{align}
			R^X(x,B)
			=& \; \E_x\left( \int^\infty_0 e^{-q t} \1_{\{U_t \in B, \; t < \tau_{b,U}^+ \wedge \tau_{0,U}^-\}} \dd t \right) + \E_x\left( \int^\infty_0 e^{-q t} \1_{\{U_t \in B, \; \tau_{b,U}^+ < t < \tau_{a,U}^+ \wedge \tau_{0,U}^-\}} \dd t \right) \notag \\
			=& \; \E_x\left( \int^\infty_0 e^{-q t} \1_{\{X_t \in B, \; t < \tau_{b,X}^+ \wedge \tau_{0,X}^-\}} \dd t \right) + \E_x\left( e^{-q \tau_{b,X}^{+}} \1_{\{\tau_{b,X}^{+} < \tau_{0,X}^{-}\}} \right) \E_b\left( \int^\infty_0 e^{-q t} \1_{\{U_t \in B, \; t < \tau_{a,U}^+ \wedge \tau_{0,U}^-\}} \dd t \right) \notag \\
			=& \; \int_{B\cap [0,b]} \left( \frac{W^{(q)}(b-y)}{W^{(q)}(b)}W^{(q)}(x) - W^{(q)}(x-y) \right) \dd y + \frac{W^{(q)}(x)}{W^{(q)}(b)} R^X(b,B) \label{eq:Thm-SMPx<bIdentity1}
		\end{align}
		where the second equality follows by recalling that $\tau_{b,X}^+ \leq T_{b<X}^+$, $\tau_{b,U}^+ \leq T_{b,U}^+$ and noticing that $\{X_t, t < T_{b,X}^+\}$ and $\{U_t, t < T_{b,U}^+\}$ have the same distribution w.r.t. $\PP_x$ when $x \in [0,b]$, and the last equality follows by using the classical Eqs.~\eqref{eq:ClassicalExitfromAbove} and \eqref{eq:ClassicalKilledPotential} of the Appendix. 
		
		Now, considering $x \in (b,a]$ and noticing that $\{Y_t, t < T_{b,Y}^-\}$ and $\{U_t, t < T_{b,U}^-\}$ have the same distribution w.r.t. $\PP_x$ for these $x$-values, we condition on $T_{b,U}^-$ and use the strong Markov property  to get
		\begin{align}
			R^Y (x,B) &= \; \E_x\left( \int^\infty_0 e^{-q t} \1_{\{U_t \in B, \; t < T_{b,U}^- \wedge \tau_{a,U}^+ \wedge \tau_{0,U}^-\}} \ \dd t \right) + \E_x\left( \int^\infty_0 e^{-q t} \1_{\{U_t \in B, \; T_{b,U}^- < t < \tau_{a,U}^+ \wedge \tau_{0,U}^-\}} \dd t \right) \notag \\
			&= \; \E_x\left( \int^\infty_0 e^{-q t} \1_{\{Y_t \in B, \; t < T_{b,Y}^- \wedge \tau_{a,Y}^+ \wedge \tau_{0,Y}^-\}} \dd t \right) + \E_x\Bigl( e^{-q T_{b,Y}^-} \1_{ \{ T_{b,Y}^- < \tau_{a,Y}^+ \wedge \tau_{0,Y}^-\}} R^X(Y_{T_{b,Y}^-},B) \Bigr) \notag \\ 
			&= \; \int_{B \cap [0,a]} \Bigl( \frac{\Ww^{(q,\lambda)}_{a}(a;y)}{\Ww^{(q,\lambda)}_{a}(a)} \Ww^{(q,\lambda)}_{x}(x) - \Ww^{(q,\lambda)}_{x}(x;y) \Bigr) \dd y \nonumber \\
			&	\quad + \int_{B\cap[0,b]} \Bigl[ \frac{W^{(q)}(b-y)}{W^{(q)}(b)} \E_x\bigl( e^{-q T_{b,Y}^-} \1_{ \{ T_{b,Y}^- < \tau_{a,Y}^+ \wedge \tau_{0,Y}^-\}} W^{(q)}(Y_{T_{b,Y}^-}) \bigr) \nonumber \\	
			&\quad 	-  \E_x\bigl( e^{-q T_{b,Y}^-} \1_{ \{ T_{b,Y}^- < \tau_{a,Y}^+ \wedge \tau_{0,Y}^-\}} W^{(q)}(Y_{T_{b,Y}^-}-y) \bigr) \Bigr] \dd y + \frac{R^X(b,B)}{W^{(q)}(b)}\E_x\bigl( e^{-q T_{b,Y}^-} \1_{ \{ T_{b,Y}^- < \tau_{a,Y}^+ \wedge \tau_{0,Y}^-\}} W^{(q)}(Y_{T_{b,Y}^-}) \bigr), 
					\label{eq:Cor1-SMPx>bIdentity-Intermediate1A}
		\end{align}
		where the last equality follows by using Lemma \ref{lem:PoissonPotentialsforResolvents} (ii) and by substituting Eq.~\eqref{eq:Thm-SMPx<bIdentity1}.
		
		Now, note that the first and third expectation in the above equation are special cases (for $y=0$) of the second expectation in this equation, and so it suffices to derive the latter one. To do this, for $x>b$, we use Eqs.~\eqref{eq:Vb-ExpectationEvaluation} and \eqref{eq:UScaleFunc} to write the second expectation of Eq.~\eqref{eq:Cor1-SMPx>bIdentity-Intermediate1A} 
		as
		\begin{align}
			\E_x &\bigl( e^{-q T_{b,Y}^-} \1_{ \{ T_{b,Y}^- < \tau_{a,Y}^+ \wedge \tau_{0,Y}^-\}} W^{(q)}(Y_{T_{b,Y}^-}-y) \bigr) \notag \\
			=& \; \frac{\Ww_{x}^{(q,\lambda)}(x)}{\Ww_{a}^{(q,\lambda)}(a)} \Bigl( \Gg_{a}^{(q,\lambda)}(a;y) - W^{(q)}(a-y)  + \color{black} \Ww_{a}^{(q,\lambda)}(a;y) \Bigr)
			- \Bigl(  \Gg_{x}^{(q,\lambda)}(x;y) -W^{(q)}(x-y)  + \color{black} \Ww_{x}^{(q,\lambda)}(x;y) \Bigr) \notag \\
			=& \; \frac{\Ww_{x}^{(q,\lambda)}(x)}{\Ww_{a}^{(q,\lambda)}(a)} \Bigl( \frac{\Ww_{a}^{(q,\lambda)}(a)}{\Ww_{b}^{(q,\lambda)}(a)} \Gg_{b}^{(q,\lambda)}(a;y) - \uU^{(q,\lambda)}_{b,a}(a;y)  + \color{black} \Ww_{a}^{(q,\lambda)}(a;y) \Bigr) \notag \\
			&- \Bigl(  \frac{\Ww_{x}^{(q,\lambda)}(x)}{\Ww_{b}^{(q,\lambda)}(a)} \Gg_{b}^{(q,\lambda)}(a;y)  - \uU^{(q,\lambda)}_{b,a}(x;y)  + \color{black} \Ww_{x}^{(q,\lambda)}(x;y) \Bigr) \notag \\
			=& \; \uU^{(q,\lambda)}_{b,a}(x;y) - \frac{\Ww_{x}^{(q,\lambda)}(x)}{\Ww_{a}^{(q,\lambda)}(a)} \uU^{(q,\lambda)}_{b,a}(a;y)  + \color{black} \Bigl( 
			\frac{\Ww_{a}^{(q,\lambda)}(a;y)}{\Ww_{a}^{(q,\lambda)}(a)} \Ww^{(q,\lambda)}_{x}(x) - \Ww_{x}^{(q,\lambda)}(x;y) \Bigr). \label{eq:Thm-ExW(q)Vb-ShiftedIdentity}
		\end{align}
		Then, by substituting the above equation into Eq.~\eqref{eq:Cor1-SMPx>bIdentity-Intermediate1A},
		\begin{align}
			R^Y (x,B)  &= \; \int_{B \cap (b,a]} \Bigl( \frac{\Ww^{(q,\lambda)}_{a}(a;y)}{\Ww^{(q,\lambda)}_{a}(a)} \Ww^{(q,\lambda)}_{x}(x) - \Ww^{(q,\lambda)}_{x}(x;y) \Bigr) \dd y \nonumber \\
			&\;\;\;\;\;+ \; \int_{B\cap[0,b]} \Bigl[ \frac{W^{(q)}(b-y)}{W^{(q)}(b)} \E_x\bigl( e^{-q T_{b,Y}^-} \1_{ \{ T_{b,Y}^- < \tau_{a,Y}^+ \wedge \tau_{0,Y}^-\}} W^{(q)}(Y_{T_{b,Y}^-}) \bigr)\notag  \\
			&\;\;\;\;\;-  \uU^{(q,\lambda)}_{b,a}(x;y) + \frac{\Ww_{x}^{(q,\lambda)}(x)}{\Ww_{a}^{(q,\lambda)}(a)} \uU^{(q,\lambda)}_{b,a}(a;y)  \Bigr] \dd y 
			+ \frac{R^X(b,B)}{W^{(q)}(b)} \E_x\bigl( e^{-q T_{b,Y}^-} \1_{ \{ T_{b,Y}^- < \tau_{a,Y}^+ \wedge \tau_{0,Y}^-\}} W^{(q)}(Y_{T_{b,Y}^-}) \bigr). \label{eq:Cor1-SMPx>bIdentityFinal}
		\end{align}
		From Eqs.~\eqref{eq:Thm-SMPx<bIdentity1} and \eqref{eq:Cor1-SMPx>bIdentityFinal}, it suffices to derive $R^X(b,B)$. To do this, we consider whether $T_{b,U}^+$ or $\tau_{a,U}^+$ occurs first and use the strong Markov property  to find that 
		\begin{align}
			R^X (b,B)
			=& \; \E_b\left( \int^\infty_0 e^{-q t} \1_{\{U_t \in B, \; t < T_{b,U}^+ \wedge \tau_{a,U}^+ \wedge \tau_{a,U}^-\}} \dd t \right) + \E_b\left( \int^\infty_0 e^{-q t} \1_{\{U_t \in B, \; T_{b,U}^+ < t < \tau_{a,U}^+ \wedge \tau_{0,U}^-\}} \dd t \right) \notag \\
			=& \; \E_b\left( \int^\infty_0 e^{-q t} \1_{\{X_t \in B, \; t < T_{b,X}^+ \wedge \tau_{a,X}^+ \wedge \tau_{0,X}^-\}} \dd t \right) + \E_b\left(  e^{-q T_{b,X}^+ } \1_{\{ T_{b,X}^+ <  \tau_{a,X}^+ \wedge \tau_{0,X}^-\}} R^Y(X_{T_{b,X}^+},B) \right) \notag \\
			=& \; \int_{B\cap [0,a]} \Bigl(\frac{\overline{W}^{(q,\lambda)}_b(b)}{\overline{W}^{(q,\lambda)}_b(a)}\overline{W}^{(q,\lambda)}_{b-y}(a-y) - \overline{W}^{(q,\lambda)}_{b-y}(b-y) \1_{\{y \in [0,b]\}} \Bigr) \dd y \notag  \\
			& + \color{black} \int_{B\cap (b,a]} \Bigl[ \frac{\Ww_{a}^{(q,\lambda)}(a;y)}{\Ww_{a}^{(q,\lambda)}(a)} \E_b\bigl(  e^{-q T_{b,X}^+ } \1_{\{ T_{b,X}^+ <  \tau_{a,X}^+ \wedge \tau_{0,X}^-\}} \Ww_{X_{T_{b,X}^+}}^{(q,\lambda)}(X_{T_{b,X}^+}) \bigr) \notag \\
			&- \E_b\bigl(  e^{-q T_{b,X}^+ } \1_{\{ T_{b,X}^+ <  \tau_{a,X}^+ \wedge \tau_{0,X}^-\}} \Ww_{X_{T_{b,X}^+}}^{(q,\lambda)}(X_{T_{b,X}^+};y) \bigr) \Bigr] \dd y \notag \\
			& + \int_{B\cap[0,b]} \Bigl[ \frac{W^{(q)}(b-y)}{W^{(q)}(b)}  \E_b\bigl(  e^{-q T_{b,X}^+ } \1_{\{ T_{b,X}^+ <  \tau_{a,X}^+ \wedge \tau_{0,X}^-\}} \E_{X_{T_b^+}}\bigl( e^{-q T_{b,Y}^-} \1_{ \{ T_{b,Y}^- < \tau_{a,Y}^+ \wedge \tau_{0,Y}^-\}} W^{(q)}(Y_{T_{b,Y}^-}) \bigr) \bigr)\notag  \\
			&-  \E_b\bigl(  e^{-q T_{b,X}^+ } \1_{\{ T_{b,X}^+ <  \tau_{a,X}^+ \wedge \tau_{0,X}^-\}} \uU_{b,a}^{(q,\lambda)}(X_{T_{b,X}^+};y) \bigr) + \frac{\uU^{(q,\lambda)}_{b,a}(a;y)}{\Ww_{a}^{(q,\lambda)}(a)} \E_b\bigl(  e^{-q T_{b,X}^+ } \1_{\{ T_{b,X}^+ <  \tau_{a,X}^+ \wedge \tau_{0,X}^-\}} \Ww_{X_{T_{b,X}^+}}^{(q,\lambda)}(X_{T_{b,X}^+}) \bigr) \Bigr] \dd y\notag  \\
			&+  \frac{R^X(b,B)}{W^{(q)}(b)} \E_b\bigl(  e^{-q T_{b,X}^+ } \1_{\{ T_{b,X}^+ <  \tau_{a,X}^+ \wedge \tau_{0,X}^-\}} \E_{X_{T_b^+}}\bigl( e^{-q T_{b,Y}^-} \1_{ \{ T_{b,Y}^- < \tau_{a,Y}^+ \wedge \tau_{0,Y}^-\}} W^{(q)}(Y_{T_{b,Y}^-}) \bigr) \bigr), 
			\label{eq:Cor1-SMPx>bAuxIdentity-Intermediate1}
		\end{align}
		where the first term in the last equality follows by using Lemma \ref{lem:PoissonPotentialsforResolvents} (i) whilst the remaining terms  of the above equation follow by using Eq.~\eqref{eq:Cor1-SMPx>bIdentityFinal}.
		
		Excluding the the fourth expectation, we note that the remaining expectations of Eq.~\eqref{eq:Cor1-SMPx>bAuxIdentity-Intermediate1} are known from either Eqs.~\eqref{eq:Thm1Proof-ExpTb+} or \eqref{eq:Thm1-FundamentalBigEbIdentity}. To compute the fourth expectation of Eq.~\eqref{eq:Cor1-SMPx>bAuxIdentity-Intermediate1}, we have from Eq.~\eqref{eq:ConvolutionOfUbaFunc} in the Appendix that
		\begin{equation}
			\lambda \int^a_b W^{(q+\lambda)}(a-z) \uU^{(q,\lambda)}_{b,a}(z;y) \dd z = \overline{W}_{b-y}^{(q,\lambda)}(a-y) - \uU_{b,a}^{(q,\lambda)}(a;y), \notag 
		\end{equation}
		and hence by using Eq.~\eqref{eq:Tb+ExpectationEvaluation} and the above equation that
		\begin{equation}
			\begin{aligned}
				\E_b\bigl(  e^{-q T_{b,X}^+ } \1_{\{ T_{b,X}^+ <  \tau_{a,X}^+ \wedge \tau_{0,X}^-\}} \uU_{b,a}^{(q,\lambda)}(X_{T_{b,X}^+};y) \bigr) = \frac{\overline{W}_{b}^{(q,\lambda)}(b)}{\overline{W}_{b}^{(q,\lambda)}(a)}\Bigr(\overline{W}_{b-y}^{(q,\lambda)}(a-y) - \uU_{b,a}^{(q,\lambda)}(a;y) \Bigl).
			\end{aligned} \label{eq:Eb-Uba(XTb+;y)}
		\end{equation}
		By noticing that $\overline{W}^{(q,\lambda)}_{b-y}(b-y) = W^{(q)}(b-y)$ for $y \in [0,b]$, using the above equation and substituting Eqs.~\eqref{eq:Thm1Proof-ExpTb+}, \eqref{eq:Thm1-FundamentalBigEbIdentity} and \eqref{eq:Eb-Uba(XTb+;y)} into Eq.~\eqref{eq:Cor1-SMPx>bAuxIdentity-Intermediate1}, we get
		\begin{align}
			R^X (b,B)
			&= \int_{B\cap (b,a]} \frac{W^{(q)}(b)}{\overline{W}^{(q,\lambda)}_b(a)}\frac{\Ww_{b}^{(q,\lambda)}(a)}{\Ww_{a}^{(q,\lambda)}(a)}\Bigl[ \Ww_{a}^{(q,\lambda)}(a;y) + \frac{\Ww_{a}^{(q,\lambda)}(a)}{\Ww_{b}^{(q,\lambda)}(a)} \bigl( \overline{W}_{b-y}^{(q,\lambda)}(a-y)-\Ww_{b}^{(q,\lambda)}(a;y)\bigr) \Bigr]\notag \\
			& + \int_{B\cap[0,b]} \Bigl[ - \frac{W^{(q)}(b-y)}{\overline{W}^{(q,\lambda)}_b(a)}\frac{\Ww_{b}^{(q,\lambda)}(a)}{\Ww_{a}^{(q,\lambda)}(a)} \uU_{b,a}^{(q,\lambda)}(a) + \frac{W^{(q)}(b)}{\overline{W}^{(q,\lambda)}_b(a)}\frac{\Ww_{b}^{(q,\lambda)}(a)}{\Ww_{a}^{(q,\lambda)}(a)} \uU_{b,a}^{(q,\lambda)}(a;y) \Bigl] \dd y\notag  \\
			&+  \frac{R^X(b,B)}{W^{(q)}(b)} \Bigl( W^{(q)}(b) -\frac{W^{(q)}(b)}{\overline{W}^{(q,\lambda)}_b(a)}\frac{\Ww_{b}^{(q,\lambda)}(a)}{\Ww_{a}^{(q,\lambda)}(a)} \uU_{b,a}^{(q,\lambda)}(a)\Bigr). 
			\label{eq:Cor1-SMPx>bAuxIdentity-IntermediateFinal}
		\end{align}
		To solve for $R^X (b,B)$ in the above equation, we first observe that $\Gg_{x}^{(q,\lambda)}(x;y) =	\gamma_{b}^{(q,\lambda)}(x;y) = W^{(q)}(x-y) - \Ww^{(q,\lambda)}_x(x;y)$ for $y >b$, and thus have from Eqs.~\eqref{eq:SecondGenScaleFunc1} and \eqref{eq:Thm1-WConvolutionAuxiliaryFuncs} that 
		\begin{equation}
			\uU^{(q,\lambda)}_{b,a}(x;y) = \Ww_{x}^{(q,\lambda)}(x;y) + \frac{\Ww_{x}^{(q,\lambda)}(x)}{\Ww_{b}^{(q,\lambda)}(a)} \bigl( \overline{W}_{b-y}^{(q,\lambda)}(a-y)-\Ww_{b}^{(q,\lambda)}(a;y)\bigr), \quad y > b, \label{eq:UbaIdentityFory>b}
		\end{equation}
		and so substitituting the above equation with $x=a$ into Eq.~\eqref{eq:Cor1-SMPx>bAuxIdentity-IntermediateFinal} yields the desired quantity
		\begin{align*}
			R^X(b,B) =& \; \frac{{W}^{(q)}(b)}{\uU^{(q,\lambda)}_{b,a}(a)} \int_{B\cap (b,a]} \uU^{(q,\lambda)}_{b,a}(a;y) \dd y + \int_{B \cap [0,b]} \Bigl( \frac{{W}^{(q)}(b)}{\uU^{(q,\lambda)}_{b,a}(a)} \uU^{(q,\lambda)}_{b,a}(a;y) - W^{(q)}(b-y)  \Bigr) \dd y \\
			=& \; \int_{B} \Bigl( \frac{{W}^{(q)}(b)}{\uU^{(q,\lambda)}_{b,a}(a)} \uU^{(q,\lambda)}_{b,a}(a;y) - W^{(q)}(b-y)  \Bigr) \dd y, \notag 
		\end{align*}
		where the last equality follows since $W^{(q)}(b-y) = 0$ for $y \in (b,a]$.
		
		Finally, by substituting the above equation into Eq.~\eqref{eq:Thm-SMPx<bIdentity1}, we derive the result for $x \in [0,b]$. For $x \in (b,a]$,  we substitute $R^X(b,B)$ along with Eq.~\eqref{eq:Thm-ExW(q)Vb-ShiftedIdentity} into Eq.~\eqref{eq:Cor1-SMPx>bIdentityFinal} and then use Eq.~\eqref{eq:UbaIdentityFory>b} to get the required result.
		
		\vspace{1ex}
		
		\noindent\upshape{(ii)} 
		Observe that $\mathbb{E}_x\left(\int_0^{\infty} e^{-q t} \mathbf{1}_{\left\{U_t \in \mathrm{~d} y, t<\tau_{a, U}^{+} \wedge \tau_{0, U}^{-}\right\}} \mathrm{d} t\right) \leq \frac{1}{q}$, and furthermore that this potential measure increases for all Borel sets $B \subset[0, a]$ as $a \rightarrow \infty$. Since it is bounded from above, its limit as $a \rightarrow \infty$ must exist, and thus, by the monotone convergence theorem,
		
		\[
		\begin{aligned}
			\mathbb{E}_x\left(\int_0^{\infty} e^{-q t} \mathbf{1}_{\left\{U_t \in B, t<\tau_{0, U}^{-}\right\}} \mathrm{d} t\right) & =\lim _{a \rightarrow \infty} \int_{B \cap[0, a]} \mathbb{E}_x\left(\int_0^{\infty} e^{-q t} \mathbf{1}_{\left\{U_t \in \mathrm{~d} y, t<\tau_{a, U}^{+} \wedge \tau_{0, U}^{-}\right\}} \mathrm{d} t\right) \\
			& =\int_{B \cap[0, \infty)} \lim _{a \rightarrow \infty}\left(\frac{\mathcal{U}_{b, a}^{(q, \lambda)}(a ; y)}{\mathcal{U}_{b, a}^{(q, \lambda)}(a ; 0)} \mathcal{U}_{b, a}^{(q, \lambda)}(x ; 0)-\mathcal{U}_{b, a}^{(q, \lambda)}(x ; y)\right) \mathrm{d} y,
		\end{aligned}
		\]
			where we have used the potential density derived in part (i) in the final step. \color{black}
			The result then follows by using Eqs.~\eqref{eq:ConclusionofUuparrow(x;y)} and \eqref{eq:ConclusionofUuparrow(y)} from the proof of Proposition \ref{label:Prop-OneSidedExitDownwards}.
			
			\vspace{1ex}
				\noindent\upshape{(iii)} Using a level invariance argument, a similar argument as in (ii) to interchange the limits and integral, and also (i), 
				\begin{equation}
					\E_x\Bigl( \int^\infty_0 e^{-q t} \1_{\{U_t \in B, \; t <  \tau_{a,U}^+\}} \dd t \Bigr) = \lim\limits_{\theta \rightarrow \infty} \int_{B} \Bigl( \frac{\uU^{(q,\lambda)}_{b+\theta,a+\theta}(x+\theta)}{\uU^{(q,\lambda)}_{b+\theta,a+\theta}(a+\theta)}\uU^{(q,\lambda)}_{b+\theta,a+\theta}(a+\theta;y+\theta) - \uU^{(q,\lambda)}_{b+\theta,a+\theta}(x+\theta;y+\theta) \Bigr) \dd y, \notag
				\end{equation}
				and so we derive the above limit by taking the limits of the terms separately. 
				
				First, by observing that $\gamma^{(q,\lambda)}_{b+\theta}(x+\theta;y+\theta) = \gamma^{(q,\lambda)}_{b}(x;y)$ and hence that $\Gg^{(q,\lambda)}_{b+\theta}(x+\theta;y+\theta) = \Gg^{(q,\lambda)}_{b}(x;y)$, we have
				\begin{equation}
					\lim\limits_{\theta \rightarrow \infty} \uU^{(q,\lambda)}_{b+\theta,a+\theta}(x+\theta;y+\theta) = W^{(q)}(x-y) -\1_{\{x>b\}} \Bigl( \gamma^{(q,\lambda)}_{b}(x;y) -  \lim\limits_{\theta \rightarrow \infty} \frac{\overline{\W}^{(q,\lambda)}_{x-b}(x+\theta)}{\Ww^{(q,\lambda)}_{b+\theta}(a+\theta)}\Gg^{(q,\lambda)}_{b}(a;y)  \Bigr), \notag
				\end{equation}
				where
				\begin{equation}
					\Ww^{(q,\lambda)}_{b+\theta}(a+\theta) = \overline{\W}^{(q,\lambda)}_{a-b}(a+\theta) + \lambda \int^a_b W^{(q+\lambda)}(a-y) \overline{\W}^{(q,\lambda)}_{y-b}(y+\theta) \dd y. \notag
				\end{equation}
				Then, by using Eqs.~\eqref{eq:LimitofWbarTheta} -- \eqref{eq:LimitofWwTheta}, it is clear that $ \uU^{(q,\lambda)\downarrow}_{b,a}(x;y) = \lim\limits_{\theta \rightarrow \infty} \uU^{(q,\lambda)}_{b+\theta,a+\theta}(x+\theta;y+\theta)$ has the form of Eq.~\eqref{eq:LimitDown-Uba(x;y)}.
				The proof is then completed by using Eq.~\eqref{eq:ConclusionofUdownarrow(x)} from the proof of Proposition \ref{label:Prop-OneSidedExitAbove}. 
				
				\vspace{1ex}
	\noindent\upshape{(iv)}		We derive the desired identities by taking the limits of the terms of the potential measure from (iii).
	Additionally, in some of the limits below, the dominated convergence theorem is applied since its usage is justified by noticing that $W^{(q+\lambda)}(a-y)/W^{(q+\lambda)}(a)  \rightarrow \color{black} \ee^{-\Phi_{q+\lambda} y}$ as $a\rightarrow \infty$.
	
	Now, since we assume $\Phi_{q}>\varphi_{q+\lambda}$, we have that $\ee^{\varphi_{q+\lambda}(a-b)}/W^{(q+\lambda)}(a) \downarrow 0$ as $a \rightarrow \infty$, and hence, by Lemma \ref{lem:LimitIdentities} (ii) and the dominated convergence theorem,
	\begin{equation}
		\lim\limits_{a \rightarrow \infty} \frac{\Z^{(q)}(a-b,\varphi_{q+\lambda})}{W^{(q+\lambda)}(a)} = 0, \quad \lim\limits_{a \rightarrow \infty} \frac{\gamma_{b}^{(q,\lambda)}(a;y)}{W^{(q+\lambda)}(a)} = 0. \notag
	\end{equation}
	Using the above limits and the dominated convergence theorem, we obtain \[\overline{\uU}_{b}^{(q,\lambda)}(x;y) = \lim\limits_{a \rightarrow \infty} \uU_{b,a}^{(q,\lambda)\downarrow}(x;y), \quad \text{ and } \quad \overline{\uU}_{b}^{(q,\lambda)}(x) = \lim\limits_{a \rightarrow \infty} \uU_{b,a}^{(q,\lambda)\downarrow}(x),
	\]
	which have the same forms as Eqs.~\eqref{eq:LimitNatural-Uba(x;y)} and \eqref{eq:LimitNatural-Uba(x)}, respectively. 
	
	Now, we observe that $\lim\limits_{a \rightarrow \infty} \uU_{b,a}^{(q,\lambda)\downarrow}(a;y)/\W^{(q)}(a) = \lim\limits_{a \rightarrow \infty} \overline{\uU}_{b}^{(q,\lambda)}(a;y)/\W^{(q)}(a)$. Then, by using Eq.~\eqref{eq:AuxLimitW-gamma} and Eq.~\eqref{eq:LimitofRatioofScaleFuncs} to notice that
	\begin{equation}
		\lim\limits_{a \rightarrow \infty} \frac{\Z^{(q)}(a-b,\varphi_{q+\lambda})}{\W^{(q)}(a)} = \frac{\lambda}{\varphi_{q+\lambda} - \varphi_q} \ee^{- \varphi_q b}, \label{eq:LimitofZFunctionforPotential}
	\end{equation}
	we conclude that $$\widetilde{{\uU}}_{b}^{(q,\lambda)}(y) = \ee^{\varphi_q b}\lim\limits_{a \rightarrow \infty} \overline{\uU}_{b}^{(q,\lambda)}(a;y)/\W^{(q)}(a)
	$$ has the same form as Eq.~\eqref{eq:LimitHarpoons-Ub(y)}.
	
	Similarly, we have that $\lim\limits_{a \rightarrow \infty} \uU_{b,a}^{(q,\lambda)\downarrow}(a)/\W^{(q)}(a) = \lim\limits_{a \rightarrow \infty} \overline{\uU}_{b}^{(q,\lambda)}(a)/\W^{(q)}(a)$, and so we need to derive the limit
	\begin{equation}
		\lim\limits_{a \rightarrow \infty} \frac{1}{\W^{(q)}(a)} \bigl(e^{\Phi_q a} - \gamma_b^{(q,\lambda)\downarrow}(a) \bigr) =  \lim\limits_{a \rightarrow \infty} \frac{1}{\W^{(q)}(a)} \Bigl( -\lambda \int_0^{\infty} \ee^{\Phi_q (b-u)}\overline{\mathbb{W}}_{a-b}^{(q,\lambda)}(a-b+u) \color{black}\mathrm{d} u \Bigr), \notag
	\end{equation}
	but it easily follows from Eq.~\eqref{eq:LimitofWBoldOverline} and the dominated convergence theorem that 
	\begin{equation}
		\lim\limits_{a \rightarrow \infty} \frac{1}{\W^{(q)}(a)} \bigl(e^{\Phi_q a} - \gamma_b^{(q,\lambda)\downarrow}(a) \bigr) =  \ee^{-\varphi_q b}\Bigl( -\lambda \int_0^{\infty} \ee^{\Phi_q (b-u)}\Z^{(q+\lambda)}(u,\varphi_q)\mathrm{d} u \Bigr).
	\end{equation}
	Then, using the above equation as well as Eq.~\eqref{eq:LimitofZFunctionforPotential}, it can be seen that \[\widetilde{{\uU}}_{b}^{(q,\lambda)} = \ee^{\varphi_q b}\lim\limits_{a \rightarrow \infty} \overline{\uU}_{b}^{(q,\lambda)}(a)/\W^{(q)}(a),\]
	has the same form as Eq.~\eqref{eq:LimitHarpoons-Ub}.			
	\end{proof}

	\subsection{Application to ruin theory}
	\label{Applications}
	
	 In this subsection, we consider the application of the above model in an insurance setting. In particular, we assume $U$ represents the risk (surplus) process of an insurer who can switch between different business strategies (represented by the L\'evy processes $X$ and $Y$, respectively) at Poissonian observation times depending on the surplus level and are interested in the time and probability of ruin - the event that the surplus drops to a negative level. 
	
	In general, the probability of ruin can be obtained as a limiting result of Proposition \ref{label:Prop-OneSidedExitDownwards} as $q\rightarrow 0$. However, for the general case of arbitrary $X$ and $Y$ the result remains in a similar form, with the $q$-scale functions replaced by their limiting counterparts, and does not offer any further insight(s). Therefore, in the remainder of this section, we will consider a specific insurance context which allows us to obtain more explicit results. 
	
Let $X$ denote the L\'evy risk process of an insurer and ${Y} = \{ {Y}_t := {X}_t - \delta t \}_{t \geq 0}$, where $\delta > 0$ represents a constant dividend rate paid to shareholders whenever the surplus ($U$) is above the level $b$. In this case, Eq.\,\eqref{eq:SDE-Main} reduces to 
\begin{equation} \label{eq:SDE-MainMain}
	U_t = x + X_t - \delta \int^t_0 \1_{\{U_{T_{N(s)}} > b\}} \dd s , \;\; U_0 = x \in \R,
\end{equation}
and the level dependent Poissonian switching can be understood as a time delay between initiating (above $b$) and withdrawing (below $b$) dividend payments, reducing the model to one similar to \cite{WLL2023}. We note here that, in addition to the arguments of Remark \ref{UniquenessOfConstruction}, the form of the above SDE implies that uniqueness can also be proved by contradiction using similar methods as in \cite{KL2010}.
	
\color{black} As before, for each $q \geq 0$, $W^{(q)}$ and $Z^{(q)}$ are the $q$-scale functions associated with $X$ and that $\mathbb{W}^{(q)}$ and $\mathbb{Z}^{(q)}$ are the $q$-scale functions associated with $Y$.  Moreover, the right inverse of the Laplace exponent of $Y$ is now $\varphi_q=\sup \{\vartheta \geq 0: \psi(\vartheta)-\delta \vartheta=q\}$. It is clear that $\delta = 0$ yields that $Y = X$ and that $\W^{(q)} = W^{(q)}$.
	
	The following corollary provides the probability of ruin for the model described above. 
	\begin{corollary} Let $0 \leq b,\lambda < \infty$, where $\lambda$ is large enough so that $\Phi_{q+\lambda} > \varphi_q$. Then,  given the assumption that $0 \leq \delta < \E(X_1)$, the ruin probability for $x \geq 0$ is given by
		\begin{linenomath}
			\begin{equation}
				\PP_x( \tau_{0,U}^- < \infty) = 1 - \frac{ (\psi'(0+) - \delta) \cdot \uU^{(0,\lambda)\uparrow}_{b}(x;0)}{  1 + \delta \lambda \int^b_0 \W^{(\lambda)}(b-y)W(y) \dd y + \Z^{(\lambda)}(b) \frac{\int^\infty_b \ee^{-\Phi_{\lambda}u} \gamma_b^{(0,\lambda)}(u) \dd u}{\int^\infty_b \ee^{-\Phi_{\lambda}u} \overline{\W}_{u-b}^{(0,\lambda)}(u) \dd u} \color{black} }, \label{Cor:ProbofRuin}
			\end{equation}
		\end{linenomath}
  where $\uU^{(0,\lambda)\uparrow}_{b}(x;y)$ is given by Eq.~\eqref{eq:lim_a-of-Uba(x)}.
	\end{corollary}
	
	\begin{proof}
		
		Observe from Proposition \ref{label:Prop-OneSidedExitDownwards} that
		\begin{linenomath}
			\begin{equation}
				\PP_x( \tau_{0,U}^- < \infty) = \lim\limits_{q \downarrow 0} \E_x\Big(e^{-q \tau_{0,U}^-} \1_{\{ \tau_{0,U}^- < \infty \}} \Big) =\lim\limits_{q \downarrow 0} \vV^{(q,\lambda)\uparrow}_{b}(x) -  \lim\limits_{q \downarrow 0} \frac{\vV^{(q,\lambda)\uparrow}_{b}}{\uU^{(q,\lambda)\uparrow}_{b}(0)} \; \lim\limits_{q \downarrow 0}\uU^{(q,\lambda)\uparrow}_{b}(x;0). \notag
			\end{equation}
		\end{linenomath}
Now, recall from \cite[Proposition 2.1]{WLL2023} that the choice of SNLPs $X_t$ and $Y_t = X_t - \delta t$ gives
		\begin{linenomath}
			\begin{equation}
				\alpha_{b}^{(q,\lambda)}(x) = -\delta q \int^x_0 \overline{\W}^{(q,\lambda)}_{x-b}(x-u)W^{(q)}(u) \dd u, \label{eq:WAN_AlphaIdentity}
			\end{equation}
		\end{linenomath}
  and so using the above equation along with Eq.~\eqref{eq:lim_a-of-Vba(x)} yields that $\lim\limits_{q \downarrow 0} \vV^{(q,\lambda)\uparrow}_{b}(x) = 1$. Furthermore, it is clear that  $\uU^{(0,\lambda)\uparrow}_{b}(x;0) = \lim\limits_{q \downarrow 0}\uU^{(q,\lambda)\uparrow}_{b}(x;0)$.
  
  Additionally, since the condition $0 \leq \delta < \E(X_1)$ implies that $\varphi_0 = 0$, the form of the denominator in Eq.~\eqref{Cor:ProbofRuin} can be found by using Eq.~\eqref{eq:lim_a-of-Uba(a)/W(q+lambda)} to take $\lim\limits_{q \downarrow 0} \uU^{(q,\lambda)\uparrow}_{b}(0)$ and then observing that $ \lim\limits_{q \downarrow 0} \Z^{(q + \lambda)}(b,\varphi_q) = \Z^{(\lambda)}(b)$ and that
  \begin{linenomath}
			\begin{equation}
				\lambda \int^b_0 \Z^{(\lambda)}(b-u) W(u) \dd u = \Z^{(\lambda)}(b) - 1 - \delta \lambda \int^b_0\W^{(\lambda)}(b-u)W(u) \dd u, \notag
			\end{equation}
		\end{linenomath}
		see \cite[Equation (A18)]{WLL2023}.

		Lastly, we observe from Eqs.~\eqref{eq:lim_a-of-Vba(a)/W(q+lambda)} and \eqref{eq:WAN_AlphaIdentity} that
			\begin{equation}
	           	\lim\limits_{q\downarrow 0} \vV^{(q,\lambda)\uparrow}_{b} =  \lim\limits_{q\downarrow 0} \ee^{\varphi_q b} \Bigl(\frac{q}{\varphi_q} + \lambda \int^b_0 \ee^{-\varphi_q y}\Z^{(q+\lambda)}(y) \dd y \Bigr)  - \lim\limits_{q\downarrow 0} \lambda \int^{b}_0  \Z^{(q+\lambda)}(b-u,\varphi_q)Z^{(q)}(u) \dd u. \notag
		\end{equation}
            As previously mentioned, $\varphi_0 = 0$ by our assumption, and $\lim_{q \rightarrow 0} \Z^{(q+\lambda)}(b-u,\varphi_q) = \Z^{(\lambda)}(b-u)$ and $\lim_{q \rightarrow 0} Z^{(q)}(x) = 1$
		which both follow by using Eq.~\eqref{eq:GeneralisedZFunc}. Finally, by noticing that $\lim_{q \rightarrow 0} \tfrac{q}{\varphi_q} = \psi'(0+) - \delta > 0$ under our assumption as well as these previous observations, we get
            \begin{equation}
				\lim\limits_{q\downarrow 0} \vV^{(q,\lambda)\uparrow}_{b} = \psi'(0+) - \delta, \notag
		\end{equation}
	which completes the proof.

	\end{proof}	
	%
	\section*{Appendix}
	The theorem below is a collection of classical fluctuation identities which have been used in the preceding text. See, for example, \cite[Chapter 8]{K2014} for the origin of these identities.
	
	\begin{lemma}[see \cite{K2014}] \label{Thm:ClassicalFluctuationIdentities}
		Let $X$ be a spectrally negative Lévy process and 
		\[
		\tau_{a,X}^{+}=\inf \left\{t>0: X_t>a\right\} \quad \text { and } \quad \tau_{0,X}^{-}=\inf \left\{t>0: X_t<0\right\} .
		\]
		\upshape{(i)}. For $q \geq 0$ and $x \leq a$
		\begin{equation}
			\mathbb{E}_x\left(\mathrm{e}^{-q \tau_{a,X}^{+}} \mathbf{1}_{\left\{\tau_{a,X}^{+}<\tau_{0,X}^{-}\right\}}\right)=\frac{W^{(q)}(x)}{W^{(q)}(a)},  \label{eq:ClassicalExitfromAbove}
		\end{equation}
		and 
		\begin{equation}
			\mathbb{E}_x\left(\mathrm{e}^{-q \tau_{0,X}^{-}} \mathbf{1}_{\left\{\tau_{0,X}^{-}<\tau_{a,X}^{+}\right\}}\right)=Z^{(q)}(x) - \frac{W^{(q)}(x)}{W^{(q)}(a)}Z^{(q)}(a) . \label{eq:ClassicalExitfromBelow}
		\end{equation}
		\upshape{(ii)}. For any $a>0, x, y \in[0, a], q \geq 0$
		\begin{equation}
			\int_0^{\infty} \mathrm{e}^{-q t} \mathbb{P}_x\left(X_t \in \mathrm{d} y, t<\tau_{a,X}^{+} \wedge \tau_{0,X}^{-}\right) \mathrm{d} t= u^{(q)}(x,a,y) \mathrm{d} y,  \label{eq:ClassicalKilledPotential}
		\end{equation}
		where 
		\begin{equation}
			u^{(q)}(x,a,y) = \frac{W^{(q)}(x) }{W^{(q)}(a)}W^{(q)}(a-y)-W^{(q)}(x-y).  \label{apost1}
		\end{equation}
		
		\noindent The following lemma is a consequence of Lemma 2.1 in \cite{LRZ2014}.
		\begin{lemma}[see \cite{LRZ2014}]
			For $p,p+q \geq 0$ and $0 \leq b \leq x \leq a$, it holds that
			\begin{equation}
				\E_x\bigl( \ee^{-(p+q) \tau_{b,X}^-} \1_{\{\tau_{b,X}^- < \tau_{a,X}^+\}} W^{(p)}(X_{\tau_{b,X}^-} - y) \bigr) = \overline{W}^{(p,q)}_{b-y}(x-y) - \frac{W^{(p+q)}(x-b)}{W^{(p+q)}(a-b)}\overline{W}^{(p,q)}_{b-y}(a-y), \quad \quad y \in [0,b),\label{eq:lem-RonnieOccupationDownJump}
			\end{equation}
			where $\overline{W}^{(p,q)}_b$ is given in Eq.~\eqref{eq:SecondGenScaleFunc1}.
			
		\end{lemma}
		
		The lemma below yields an identity required for the derivation of the fluctuation identities.
		\begin{lemma}
			Let $0 < b \leq a$ and $0 < \lambda < \infty$. Then, for $q \geq 0$ and $0 \leq x,y \leq a$, we have
			\begin{equation}
				\lambda \int^a_b W^{(q+\lambda)}(a-y) \vV_{b,a}^{(q,\lambda)}(y) \dd y = \overline{Z}^{(q,\lambda)}_{b}(a) - \vV_{b,a}^{(q,\lambda)}(a), \label{eq:ConvolutionOfVbaFunc}
			\end{equation}
			and
			\begin{equation}
				\lambda \int^a_b  W^{(q+\lambda)}(a-z) \uU^{(q,\lambda)}_{b,a}(z;y) \dd z =  \overline{W}_{b-y}^{(q,\lambda)}(a-y) - \uU^{(q,\lambda)}_{b,a}(a;y), \label{eq:ConvolutionOfUbaFunc}
			\end{equation}
			where $\uU^{(q,\lambda)}_{b,a}$ and $\vV^{(q,\lambda)}_{b,a}$   are defined in Eqs.\eqref{eq:UScaleFunc} and \eqref{eq:VScaleFunc}, respectively.
			
		\end{lemma}
		
		\begin{proof}
			First, we prove Eq.~\eqref{eq:ConvolutionOfVbaFunc}. By substituting the form of Eq.~\eqref{eq:VScaleFunc} into the integral, we get that
			\begin{align}
				\lambda \int^a_b W^{(q+\lambda)}(a-y) \vV_{b,a}^{(q,\lambda)}(y) \dd y =& \; \lambda \int^a_b W^{(q+\lambda)}(a-y) Z^{(q)}(y) \dd y - \lambda \int^a_b W^{(q+\lambda)}(a-y)\mathcal{A}_{y}^{(q, \lambda)}(y) \dd y \notag \\
				& + \color{black} \frac{\mathcal{A}_{b}^{(q, \lambda)}(a)}{\mathcal{W}_{b}^{(q, \lambda)}(a)} \lambda \int^a_b W^{(q+\lambda)}(a-y) \mathcal{W}_{y}^{(q, \lambda)}(y) \dd y \notag \\
				=& \; \overline{Z}_b^{(q,\lambda)}(a) - Z^{(q)}(a) + \Aa_{a}^{(q,\lambda)}(a) - \Aa_{b}^{(q,\lambda)}(a) \notag \\
				& + \color{black} \frac{\mathcal{A}_{b}^{(q, \lambda)}(a)}{\mathcal{W}_{b}^{(q, \lambda)}(a)} \bigl( \Ww_{b}^{(q, \lambda)}(a) - \Ww_{a}^{(q, \lambda)}(a) \bigr) \notag \\
				=& \; \overline{Z}_b^{(q,\lambda)}(a) - \vV_{b,a}^{(q, \lambda)}(a), \notag
			\end{align}
			where the second equality has used  Eqs.~\eqref{eq:SecondGenScaleFunc2}, \eqref{eq:Thm1-WConvolutionAuxiliaryFuncs} and \eqref{eq:Thm1-AConvolutionAuxiliaryFuncs}, and the last equality uses Eq.~\eqref{eq:VScaleFunc}.
			\\
			\indent One can derive Eq.~\eqref{eq:ConvolutionOfUbaFunc} similarly by substituting the form of \eqref{eq:UScaleFunc} to get 
			\begin{align}
				\lambda \int^a_b W^{(q+\lambda)}(a-z) \uU^{(q,\lambda)}_{b,a}(z;y) \dd z 
				=& \; \lambda \int^a_b  W^{(q+\lambda)}(a-z) W^{(q)}(z-y) \dd z - \lambda \int^a_b W^{(q+\lambda)}(a-z) \Gg^{(q,\lambda)}_{z}(z;y) \dd z \notag \\
				&+ \frac{\Gg^{(q,\lambda)}_{b}(a;y)}{\Ww^{(q,\lambda)}_{b}(a)}\lambda \int^a_b W^{(q+\lambda)}(a-z) \Ww^{(q,\lambda)}_{z}(z) \color{black} \dd z \notag \\
				=& \; \overline{W}_{b-y}^{(q,\lambda)}(a-y) -W^{(q)}(a-y) + \Gg^{(q,\lambda)}_{a}(a;y) - \Gg^{(q,\lambda)}_{b}(a;y) \notag \\
				&+ \frac{\Gg^{(q,\lambda)}_{b}(a;y)}{\Ww^{(q,\lambda)}_{b}(a)} \Bigl( \Ww^{(q,\lambda)}_{b}(a) - \Ww^{(q,\lambda)}_{a}(a) \Bigr) \notag \\
				=& \; \overline{W}_{b-y}^{(q,\lambda)}(a-y) - \uU_{b,a}^{(q,\lambda)}(a;y) , \notag 
			\end{align}
			where the second equality is obtained by using Eqs.~\eqref{eq:SecondGenScaleFunc1}, \eqref{eq:Thm1-WConvolutionAuxiliaryFuncs} and \eqref{eq:Thm1-GConvolutionAuxiliaryFuncs}.
		\end{proof}
		
	\end{lemma}
	\section*{Acknowledgement}
The authors are grateful to the anonymous referees for their constructive comments and suggestions that have improved the content and presentation of this paper. 	We are also grateful  to Ronnie Loeffen for suggesting the approach considered in Section \ref{ConstructionofSwitching}.

	\bibliography{Apostolos.bib}
	\bibliographystyle{abbrv}

\end{document}